\documentclass[10pt,reqno]{amsart}
\usepackage{amssymb,mathrsfs,graphicx,extpfeil,amsmath,bbm}
\usepackage{ifthen,latexsym,float,colortbl}
\usepackage[margin=1in]{geometry}
\usepackage{caption}
\usepackage{sidecap}
\usepackage{rotating,dsfont,stackengine}
\usepackage{enumerate}
\usepackage{colortbl}
\definecolor{black}{rgb}{0.0, 0.0, 0.0}
\definecolor{red}{rgb}{1.0, 0.5, 0.5}
\provideboolean{shownotes} % define a boolean variable
\setboolean{shownotes}{true} % defined as ``true'' or ``false''
\newcommand{\margnote}[1]{
\ifthenelse{\boolean{shownotes}}%
{\marginpar{\raggedright\tiny\texttt{#1}}}%
{}%
}
\newcommand{\hole}[1]{
\ifthenelse{\boolean{shownotes}}%
{\begin{center} \fbox{ \rule {.25cm}{0cm} \rule[-.1cm]{0cm}{.4cm}
\parbox{.85\textwidth}{\begin{center} \texttt{#1}\end{center}} \rule
{.25cm}{0cm}}\end{center}} {} }

\graphicspath{{pics/}}
\title[From BGK-alignment model to the pressured Euler-alignment system ]
{From BGK-alignment model to the pressured Euler-alignment system with singular communication weights}

\author[Choi]{Young-Pil Choi}
\address[Young-Pil Choi]{\newline Department of Mathematics\newline
Yonsei University, 50 Yonsei-Ro, Seodaemun-Gu, Seoul 03722, Republic of Korea}
\email{ypchoi@yonsei.ac.kr}

\author[Hwang]{Byung-Hoon Hwang}
\address[Byung-Hoon Hwang]{\newline Department of Mathematics Education\newline
Sangmyung University, 20 Hongjimun 2-gil, Jongno-Gu, Seoul 03016, Republic of Korea}
\email{bhhwang@smu.ac.kr}

\numberwithin{equation}{section}

\newtheorem{theorem}{Theorem}[section]
\newtheorem{lemma}{Lemma}[section]

\newtheorem{proposition}{Proposition}[section]
\newtheorem{remark}{Remark}[section]
\newtheorem{definition}{Definition}[section]

\newcommand{\R}{\mathbb R}
\newcommand{\K}{\kappa} 
\newcommand{\N}{\mathbb N}

\newcommand{\om}{\Omega}
\newcommand{\ls}{\lesssim}

\newcommand{\T}{\mathbb T}

\newcommand{\mc}{\mathcal C}

\newcommand{\bq}{\begin{equation}}
\newcommand{\eq}{\end{equation}}
\newcommand{\e}{\varepsilon}
\newcommand{\lt}{\left}
\newcommand{\rt}{\right}
\newcommand{\lal}{\langle}
\newcommand{\ral}{\rangle}
\newcommand{\pa}{\partial}

\newcommand{\mh}{\mathcal{H}}
\newcommand{\me}{\mathcal{E}}

\newcommand{\into}{\int_\om}
\newcommand{\intoo}{\iint_{\om \times \om}}

\newcommand{\intr}{\int_{\R^d}}
\newcommand{\intt}{\int_{\T^d}}
\newcommand{\inttr}{\iint_{\R^d \times \R^d}}
\newcommand{\intor}{\iint_{\om \times \R^d}}

\DeclareMathOperator*{\esssup}{ess\,sup}
\makeatletter
\def\moverlay{\mathpalette\mov@rlay}
\def\mov@rlay#1#2{\leavevmode\vtop{%
   \baselineskip\z@skip \lineskiplimit-\maxdimen
   \ialign{\hfil$\m@th#1##$\hfil\cr#2\crcr}}}
\newcommand{\charfusion}[3][\mathord]{
    #1{\ifx#1\mathop\vphantom{#2}\fi
        \mathpalette\mov@rlay{#2\cr#3}
      }
    \ifx#1\mathop\expandafter\displaylimits\fi}
\makeatother

\begin{document}
%%%%%%%%%%%%%%%%
\allowdisplaybreaks

\date{\today}

%\subjclass[]{35B40, 35Q35, 35B40}

\begin{abstract} 
This paper is devoted to a rigorous derivation of the  isentropic Euler-alignment system with singular communication weights $\phi_\alpha(x) = |x|^{-\alpha}$ for some $\alpha > 0$. We consider a kinetic BGK-alignment model consisting of a kinetic BGK-type equation with a singular Cucker-Smale alignment force. By taking into account a small relaxation parameter, which corresponds to the asymptotic regime of a strong effect from BGK operator, we quantitatively derive the isentropic Euler-alignment system with pressure $p(\rho) = \rho^\gamma$, $\gamma = 1 + \frac2d$ from that kinetic equation.
\end{abstract}

\maketitle \centerline{\date}

\tableofcontents

%%%%%%%%%%%%%%%%%%%%%%%%%%%%%%%%%%%%%%%%%%%%%%%%%%%%%%%%%%%%%%%%%%%%%%%%%%%%%%%%%5
%
%
%                        Section: Introduction 
%
%
%%%%%%%%%%%%%%%%%%%%%%%%%%%%%%%%%%%%%%%%%%%%%%%%%%%%%%%%%%%%%%%%%%%%%%%%%%%%%%%%%
\section{Introduction}
The main purpose of this work is to rigorously and quantitatively derive the following isentropic Euler equations with nonlocal velocity alignment forces, often referred to as the isentropic Euler-alignment system (in short, EAS):
\begin{align}\label{main_pE}
\begin{aligned}
&\pa_t \rho + \nabla_x \cdot (\rho u) = 0, \quad (x,t) \in \om \times \R_+,\cr
&\pa_t (\rho u) + \nabla_x \cdot (\rho u \otimes u +  \rho^\gamma \mathbb{I}_d ) = -  \rho \int_\om \phi_\alpha(x-y)(u(x) - u(y))\rho(y)\,dy,
\end{aligned}
\end{align}
subject to the smooth initial data 
\[
(\rho(x,t),u(x,t))|_{t=0} =: (\rho_0(x),u_0(x)),\quad x\in\Omega,
\]
where $\Omega$ is a spatial domain, either $\T^d$ or $\R^d$ with the space dimension $d \geq 1$, and $\gamma = 1 + \frac2d$. Here $\rho = \rho(x,t)$ represents the density and $u=u(x,t)$ is the velocity at position $x \in \om$ and time $t \in \R_+$. The right-hand side of the momentum equation in \eqref{main_pE} is the velocity alignment force, where $\phi_\alpha: \R^d \to \R_+$ is called the communication weight function. Throughout this paper, we assumed that $\phi$ is  radially symmetric, i.e., $\phi(x) = \hat\phi(|x|)$ for some $\hat\phi : \R_+  \to \R_+$. By abuse of notation, we use $\phi(r):=\hat{\phi}(r)$ for simplicity. In the current work, we deal with the singular communication weight function which has the form of
\bq\label{inf}
\phi_\alpha(r) = \frac{1}{r^\alpha}, \quad \alpha > 0.
\eq
Note that when $\om = \T^d$, the singular communication weight $\phi_\alpha : [0,2\pi) \to \R$ can be chosen as 
\bq\label{phi_per}
\phi_\alpha(r)  = \left\{ \begin{array}{ll}
\displaystyle \frac1{r^\alpha}  & \textrm{if $r \in (0,\pi]$}\\[2mm]
\displaystyle \frac1{(2\pi -r)^\alpha}  & \textrm{if $r \in (\pi, 2\pi)$}\\[2mm]
  0 & \textrm{if $r = 0$}
  \end{array} \right..
\eq

The EAS arises as a macroscopic description of the celebrated Cucker-Smale model \cite{CS07} which is a Newton-type microscopic model for an interacting many-body system exhibiting a flocking phenomenon. For that reason, EAS is also often called the hydrodynamic Cucker-Smale model. In \cite{CS07}, the regular  function which has the form of 
\[%\bq\label{reg}
\phi_\alpha(r) = \frac{1}{(1+r^2)^{\frac\alpha2}}, \quad \alpha > 0.
\]%\eq
is considered, and sufficient conditions for initial data  and $\alpha > 0$ leading to the velocity alignment behavior of solutions are analyzed. Later, collision avoidance behaviors of solutions for the Cucker-Smale model with singular weights are discussed in \cite{ACHL12, CCMP17, MMPZ19, P14}. In particular, it is observed that the Cucker-Smale flocking particles avoid collisions regardless of the initial data when $\alpha \geq 1$. We refer to \cite{CCP17, CFRT10,  CHL17,HL09, HT08, MMPZ19, S21, TT14} and references therein for the flocking estimates of solutions to particle, kinetic, and hydrodynamic descriptions of Cucker-Smale model and its variants.

%%%%%%%%%%%%%%%%%%%%%%%%%%%%%%%%%%%%%%%%%%%%%%%%%%%%%%%%%%%%%%
%
%
%.  
%.  
%
%
%%%%%%%%%%%%%%%%%%%%%%%%%%%%%%%%%%%%%%%%%%%%%%%%%%%%%%%%%%%%%%%%%%%%
\subsection{Derivation of EAS from mesoscopic descriptions} The mesoscopic description of Cucker-Smale model is given by the following Vlasov-type equation:
\bq\label{kin}
\pa_t f + v\cdot\nabla_x f  + \nabla_v \cdot \lt(F_{\phi_\alpha}[f]f\rt)  = 0, 
\eq
where $f = f(x,v,t)$ stands for the one-particle distribution function at $(x,v) \in \Omega \times \R^d$ and time $t > 0$.  
Here $F_{\phi_\alpha}$ represents the velocity alignment force field:
\[
F_{\phi_\alpha}[f](x,v) := \iint_{\om \times \R^d} \phi_\alpha(x-y) (w-v)f(y,w)\,dydw. 
\]
For the kinetic Cucker-Smale model \eqref{kin} with the singular communication weight \eqref{inf}, the global-in-time existence of measure-valued solutions is studied in \cite{PP18} when $\alpha < \frac12$, and the local-in-time existence and uniqueness of $L^p$ solutions are obtained in \cite{CCH14} when $\alpha < d-1$. The uniform-in-time mean-field limit from particle systems in one dimension is recently discussed in \cite{CZ21} when $\alpha < 1$.

Formally, by introducing the macroscopic density $\rho_f(x,t)$ and bulk velocity $u_f(x,t)$ as 
\[
\rho_f(x,t):=\intr f(x,v,t)\,dv \quad \mbox{and} \quad u_f(x,t):= \frac{\intr vf(x,v,t)\,dv}{\rho_f(x,t)},
\]
respectively, we derive from \eqref{kin} that
\begin{align*}%\label{formal}
\begin{aligned}
&\pa_t \rho_f + \nabla_x \cdot (\rho_f u_f) = 0,\cr
&\pa_t (\rho_f u_f) + \nabla_x \cdot (\rho_f u_f \otimes u_f) + \nabla_x \cdot \lt( \intr (v - u_f) \otimes (v - u_f) f\,dv\rt) \cr
&\quad = -  \rho_f \int_\om \phi_\alpha(x-y)(u_f(x) - u_f(y))\rho_f(y)\,dy.
\end{aligned}
\end{align*}

There are several formal ways of closing the momentum equations above. 
\subsubsection{Pressureless EAS}

Taking into account the mono-kinetic ansatz $f(x,v) = \rho_f(x) \delta_{u_f}(v)$ leads to
\[
\intr (v - u_f) \otimes (v - u_f) f\,dv = 0.
\]
Thus, this closure assumption results in the following pressureless EAS:
\begin{align}\label{pEAS}
\begin{aligned}
&\pa_t \rho_f + \nabla_x \cdot (\rho_f u_f) = 0,\cr
&\pa_t (\rho_f u_f) + \nabla_x \cdot (\rho_f u_f \otimes u_f)   = -  \rho_f \int_\om \phi_\alpha(x-y)(u_f(x) - u_f(y))\rho_f(y)\,dy.
\end{aligned}
\end{align}
In \cite{CCJ21,FK19}, the strong local alignment forces are considered in the kinetic equation \eqref{kin}, i.e., $\frac1\e \nabla_v \cdot ((v-u)f)$ on the right hand side of \eqref{kin}, and it is found that the pressureless EAS \eqref{pEAS} can be rigorously derived when $\e \to 0$. Recently,  the rigorous derivation of the pressureless EAS directly from the particle Cucker-Smale model, not via its kinetic formulation, by means of mean-field limit is investigated in \cite{CC21}.  In those works, the communication weight function is assumed to be regular enough, for instance $\phi_\alpha \in W^{1,\infty}(\R^d)$. For the system \eqref{pEAS} with regular communication weight functions, the critical threshold phenomenon is first proved in \cite{TT14}, and later a sharp dichotomy of initial configurations, either subcritical initial data which evolve into global strong solutions, or supercritical initial data which will blow up in finite time, is analyzed in \cite{CCTT16}. The critical thresholds phenomena in \eqref{pEAS} with the singular communication weights are studied in \cite{T20}. We also refer to \cite{LS21,LSpre, LS19} for the global well-posedness theory and long-time dynamics of solutions to the pressureless EAS \eqref{pEAS}.

\subsubsection{Pressured EAS}\label{ssec_pre}
In order to have the pressure term in the momentum equations, at the formal level, we can deal with the following three different closure assumptions on $f$ depending on $\gamma \in [1,1 + \frac 2d]$. More precisely, it follows from \cite{BV05, Bou99} that 
\[
\intr (v - u_f) \otimes (v - u_f) f\,dv = \rho_f^\gamma \mathbb{I}_d, \quad \gamma \in \lt[1, 1 + \frac2d \rt],
\]
where
\bq\label{pre}
f = \left\{ \begin{array}{ll}
  \displaystyle \frac{\rho_f}{(2\pi)^{d/2}} \exp\lt(-\frac{|u_f-v|^2}{2} \rt) & \textrm{for $\gamma = 1$}\\[3mm]
  \displaystyle c_{\gamma,d}\left(\frac{2\gamma}{\gamma-1}\rho_f^{\gamma-1}-|v-u_f|^2\right)^{n/2}_+& \textrm{for $\gamma \in (1, 1 + \frac2d)$}\\[3mm]
  \displaystyle   \mathds{1}_{|u_f-v|^d\le c_d\rho_f} & \textrm{for $\gamma = 1 + \frac2d$}
  \end{array} \right.
\eq
with
\bq\label{para}
n =\frac{2}{\gamma-1}-d, \quad c_d=\frac{d}{|\mathbb{S}_d|}, \quad \mbox{and} \quad c_{\gamma,d} =\left(\frac{2\gamma}{\gamma-1}\right)^{-1/(\gamma-1)}\frac{\Gamma\left(\frac{\gamma}{\gamma-1}\right)}{\pi^{d/2}\Gamma(n/2+1)}. 
\eq
Note that the uniform distribution-type closure assumption, which corresponds to the case $\gamma = 1 + \frac2d$, in \eqref{pre} gives the pressure term in our main system \eqref{main_pE}.

The local Maxwellian-type closure assumption can be justified by considering the nonlinear Fokker-Planck term on the right hand side of \eqref{kin}. More precisely, when the communication weight $\phi_\alpha$ is bounded,  the isothermal EAS, i.e. \eqref{main_pE} with $\gamma=1$, is rigorously derived in \cite{KMT15} (see also \cite{CCJ21}) from the following kinetic Fokker-Planck-alignment model:
 \bq\label{FKa}
 \pa_t f^\e + v\cdot\nabla_x f^\e  + \nabla_v \cdot \lt(F_{\phi_\alpha}[f^\e]f^\e\rt)  = \frac{1}{\e}\mathcal{N}_{{\rm FP}}[f^\e], 
 \eq
 where 
$\mathcal{N}_{{\rm FP}}$ denotes the nonlinear Fokker-Planck operator \cite{V02} given by 
\[
 \mathcal{N}_{{\rm FP}}[f](x,v) := \nabla_v \cdot ( (v - u_f)f +  \nabla_v f).
\]
Note that the asymptotic regime for \eqref{FKa} corresponds to a strong effect from the nonlinear Fokker-Planck term. Note that we expect $f^\e \sim  \frac{\rho_{f^\e}}{(2\pi)^{d/2}} \exp\lt(-\frac{|u_{f^\e}-v|^2}{2} \rt)$ for $\e \gg 1$, thus if $\rho_{f^\e} \to \rho$ and $u_{f^\e} \to u$ in some sense as $\e \to 0$, then we find
\[
\intr (v - u_{f^\e}) \otimes (v - u_{f^\e}) f^\e\,dv \to \rho \mathbb{I}_d \quad \mbox{as} \quad \e \to 0.
\]
In \cite{CCJ21, KMT15}, the almost everywhere convergences of $\rho_{f^\e}$, $\rho_{f^\e}u_{f^\e}$, and $f^\e$ toward $\rho$, $\rho u$, and $M_{\rho,u}$, respectively, are obtained as $\e \to 0$. More recently, the singular communication weight case with $\alpha \in (0,1)$ is also covered in \cite{CKpre22} by using the relative entropy method with careful analyses of singular integrals.

%%%%%%%%%%%%%%%%%%%%%%%%%%%%%%%%%%%%%%%%%%%%%%%
%
%
%
%
%
%
%%%%%%%%%%%%%%%%%%%%%%%%%%%%%%%%%%%%%%%%%%%%%%%

%\subsection{BGK-alignment model}

%%%%%%%%%%%%%%%%%%%%%%%%%%%%%%%%%%%%%%%%%%%%%%%%%%%%%%%%%%%%%%
%
%
%.  
%.  
%
%
%%%%%%%%%%%%%%%%%%%%%%%%%%%%%%%%%%%%%%%%%%%%%%%%%%%%%%%%%%%%%%%%%%%%
\subsection{Outline of methodology}

In order to derive the isentropic EAS \eqref{main_pE} from kinetic descriptions, motivated from \cite{BB00, CKpre22}, we shall need to consider the following kinetic BGK-alignment model:
\begin{align}\label{main_eq}
\begin{aligned}
&\pa_t f^\e + v\cdot\nabla_x f^\e  + \nabla_v \cdot \lt(F_{\phi^\varepsilon_\alpha}[f^\e]f^\e\rt)  = \frac{1}{\e}Q[f^\e], 
\end{aligned}
\end{align}
subject to the initial data:
\[
f^\e(x,v,t)|_{t=0} =: f^\e_0(x,v), \quad (x,v) \in \om \times \R^d.
\]
Here $\phi^\e_\alpha: \R^d \to \R_+$ is a regularized communication weight function defined as
\bq\label{r_inf}
\phi^\e_\alpha(r) = \phi_\alpha(\sqrt{\e^\beta + r^2}) \quad\mbox{with}\quad \alpha >0,\quad \beta>0
\eq
and
$Q[f]=Q[f](x,v,t)$ denotes the BGK-type relaxation operator given by 
\[
Q[f] := M[f] - f, \quad \mbox{where}  \quad M[f] =   \mathds{1}_{|u-v|^d\le c_d\rho}.
\]
It is clear that at the formal level, $\phi^\e_\alpha \to \phi_\alpha$ and $Q[f^\e] = M[f^\e] - f^\e \to 0$ as $\e \to 0$, thus if $\rho_{f^\e} \to \rho$, $u_{f^\e} \to u$, and $f^\e \to   \mathds{1}_{|u-v|^d\le c_d\rho}$ as $\e\to 0$, then by the discussion above we have that $\rho$ and $u$ satisfy our main system, isentropic EAS \eqref{main_pE}. We notice that a similar scaling for the weight function $\phi_\alpha$ \eqref{r_inf} is also dealt with in \cite{PS17} to study hyperbolic limits of kinetic Cucker-Smale model with singular weights. Throughout this paper, we assume that $f^\e$ is a probability density, i.e., $\|f^\e(\cdot,\cdot,t)\|_{L^1} = 1$ for $t\geq 0$ and $\e>0$, since the total mass is preserved in time. 

To the authors' best knowledge, a rigorous derivation of the isentropic EAS, even with the regular communication weights, has not been established yet. For that reason, the main purpose of this work is to study the derivation of the isentropic EAS rigorously and quantitatively. We would like to stress that our argument can be directly applied to the case with regular communication weights. We also remark that the hydrodynamic limit is studied in \cite{BV05} when the force field is a given function as $F = F(x)$ with $F \in L^\infty(\om)$.

Our strategy is based on the relative entropy method, also often referred to as modulated energy method, which relies on the weak-strong uniqueness principle for systems of conservation laws \cite{Daf79, DiP79}. Later, it has been successfully applied to various hydrodynamic limit problems \cite{BV05, Bou99, Bre00, CCJ21, CJ21, MV08, SR09, Yau91}, etc. To make use of relative entropy method, we rewrite the pressured EAS \eqref{main_pE} as a conservative form:
\[%\bq\label{sys_cons}
\pa_t U + \nabla \cdot A(U) = F(U),
\]%\eq
where 
\[
 U := \begin{pmatrix}
\rho \\
m 
\end{pmatrix} 
\quad \mbox{with} \quad m := \rho u, \quad
A(U) := \begin{pmatrix}
m  \\
\displaystyle \frac{m \otimes m}{\rho} + \rho^\gamma \mathbb{I}_d
\end{pmatrix},
\]
and
\[
F(U) := \begin{pmatrix}
0 \\
\displaystyle   \rho \into \phi_\alpha(x-y)(u(y) - u(x))\rho(y)\,dy  \end{pmatrix}.
\]
The free energy of the above system is given by
\[%\bq\label{ent_mac}
E(U) := \frac{|m|^2}{2\rho} + \frac{1}{\gamma-1}\rho^\gamma.
\]%\eq
We then define the relative entropy functional $\me$ between two states of the system $U$ and $\bar U$ as follows.
\begin{align}\label{def_rel}
\begin{aligned}
\me(\bar U|U) &:= E(\bar U) - E(U) - DE(U)(\bar U-U) \quad \mbox{with} \quad \bar U := \begin{pmatrix}
        \bar\rho \\
        \bar m \\
    \end{pmatrix}, \quad \bar m = \bar\rho \bar u,\cr
    &= \frac{\bar\rho}{2}|\bar u - u|^2 + \mh(\bar\rho| \rho),
\end{aligned}
\end{align}
where $D E(U)$ denotes the derivative of $E$ with respect to $U$, and
\bq\label{def_h}
\mh(\bar\rho | \rho):=\frac{1}{\gamma-1}(\bar\rho^\gamma - \rho^\gamma+\gamma(\rho - \bar\rho)\rho^{\gamma-1}).
\eq
We now let $U^\e = (\rho^\e,m^\e = \rho^\e u^\e)$ be the macroscopic quantity corresponding to a solution of the BGK-alignment model \eqref{main_eq} and estimate a quantitative bound on 
\[
\into \me(U^\e|U)\,dx,
\]
where
\[
\rho^\e := \intr f^\e\,dv \quad \mbox{and} \quad \rho^\e u^\e = \intr vf^\e\,dv.
\]
It follows from \eqref{main_eq} that $U^\e = (\rho^\e,m^\e = \rho^\e u^\e)$ satisfies 
\[
\pa_t U^\e + \nabla \cdot A(U^\e) - F^\e(U^\e) = \begin{pmatrix} 0 \\ \nabla \cdot \lt( \displaystyle \intr (u^\e \otimes u^\e - v \otimes v)f^\e\,dv + C_d(\rho^\e)^\gamma \mathbb{I}_d \rt)  \end{pmatrix},
\]
where $C_d$ is given by
\bq\label{c_d}
C_d := \frac{|\mathbb{S}_{d-1}|}{d(d+2)} \lt(\frac{d}{|\mathbb{S}_{d-1}|} \rt)^{\frac{d+2}d}.
\eq
Here the source term $F^\e=F^\e(U^\e)$ is given as 
\[
F^\e(U^\e) := \begin{pmatrix}
	0 \\
	\displaystyle   \rho^\e \into \phi^\e_\alpha(x-y)(u^\e(y) - u^\e(x))\rho^\e(y)\,dy  \end{pmatrix}.
\]
Obviously, the main difficulty in estimating the relative entropy arises from the singular communication weight $\phi_\alpha$. As mentioned above, for $\alpha \in (0,1]$, the isothermal EAS is rigorously derived from the kinetic Fokker-Planck-alignment model \eqref{FKa} recently in \cite{CKpre22}. In that work, the bound estimate $\phi_\alpha(x-y)|u(x) - u(y)| \leq {\rm Lip}(u) |x-y|^{1-\alpha} \ls 1$ for $|x - y| \ls 1$ is heavily used, thus it seems hard to extend to the case with $\alpha > 1$.  We mainly follow the argument used in \cite{CKpre22}, however, to cover the more singular regime $\alpha > 1$, we consider the BGK-alignment model \eqref{kin} and have a better control of $\rho^\e - \rho$ from the relative entropy functional $\mh(\rho^\e|\rho)$ appeared in \eqref{def_h}. This together with a careful analysis of singular integrals enables us to close the relative entropy estimate, and the almost everywhere convergences of $\rho^\e$ and $\rho^\e u^\e$ to $\rho$ and $\rho u$ are obtained, respectively.

As stated above, we employ the relative entropy argument, and this requires the existence of weak solutions to the equation \eqref{main_eq} and regular solutions to the limiting system \eqref{main_pE}. Since local existence theories for Euler-type equations have been well developed, we omit the details of proof for the uniqueness and existence of regular solutions to the system \eqref{main_pE} and prove the existence of $L^1_+ \cap L^\infty$-solutions to the BGK-alignment model \eqref{main_eq}. Indeed, when $\om =\T^d$, the local-in-time well-posedness for the isentropic EAS \eqref{main_pE} can be obtained by using a similar argument as in \cite{CTT}, see also \cite{C19,CDS20}. In the whole space case, we can use almost the same argument as in \cite{CKpre22} to establish the local well-posedness.

Concerning the existence of weak solutions to \eqref{main_eq}, in the one dimensional case, the global existence of $L^1_+$-solutions to the equation \eqref{main_eq} when $F \equiv 0$ and $M[f]$ is given as the positive function in \eqref{pre}, see \eqref{eq_pos} below, is established in \cite{BB00}. In a recent work \cite{CHpre}, global existence of $L^1_+ \cap L^\infty$-solutions to the equation \eqref{main_eq} is obtained in dimension $d\geq 1$ when $F \equiv 0$. In that work, the positive function in \eqref{pre} is also dealt with. In our case, even though the force field $F_{\phi^\e_\alpha}$ is not singular for $\e>0$, it produces new difficulties that are not observed in previous literature due to its nonlinearity and nonlocality.  To overcome that difficulties, inspired by \cite{CHpre}, we regularize the macroscopic observables $\rho$ and $u$ in $M[f]$, and we further regularize $M[f]$ itself to remove its discontinuity. We then linearize that regularized equation. For that regularized and linearized equation, we provide some uniform bound estimates and Cauchy estimate for approximate solutions. In order to show the Cauchy estimate, we use a velocity-weighted $L^1$ space which enables us to obtain an appropriate information from the support of $M[f]$. However, our force field $F_{\phi_\alpha} = \phi_\alpha * (\rho_f u_f) - v (\phi_\alpha * \rho_f)$ has a linear growth with respect to $v$, and this causes a 
growth of $v$-weight for $f$. To control the velocity growth of $f$, motivated from \cite{CJpre}, we estimate the approximate solutions in a weighted $L^\infty$ space by exponential velocity-weight. By combining this and extracting a proper dissipative effect from the force field, we are able to close the Cauchy estimate. This gives the existence of solutions to the regularized equation. We then finally use weak and strong compactness theorems to pass to the limit in the regularized equation.

%%%%%%%%%%%%%%%%%%%%%%%%%%%%%%%%%%%%%%%%%%%%%%%%%%%%%%%%%%%%%%
%
%
%.  
%.  
%
%
%%%%%%%%%%%%%%%%%%%%%%%%%%%%%%%%%%%%%%%%%%%%%%%%%%%%%%%%%%%%%%%%%%%%
\subsection{Main results}

We first introduce our notion of weak solutions to equation \eqref{main_eq} and then state the existence theorem. 

 \begin{definition}\label{def_weak} For a given $T>0$, we say that $f^\e$ is a weak solution to \eqref{main_eq} if the following conditions are satisfied: 
	\begin{itemize}
		\item[(i)] $f^\e \in L^\infty(0,T; L^1_+ \cap L^\infty(\om \times \R^d))$ and 
		\item[(ii)] for all $\eta \in \mc^1_c(\om \times \R^d \times [0,T])$ with $\eta(x,v,T) = 0$,
	\begin{align*}
	&- \iint_{\om\times\R^d} f^\e_0 \eta(x,v,0)\,dxdv - \int_0^T \iint_{\om\times\R^d} f^\e (\pa_t \eta + v \cdot \nabla_x \eta+F_{\phi^\e_\alpha}[f^\e]\cdot\nabla_v\eta)\,dxdvdt \cr
	&\qquad= \int_0^T \iint_{\om\times\R^d} \lt(M[f^\e] - f^\e\rt)\eta\,dxdvdt.
	\end{align*}
	\end{itemize}
\end{definition}

\begin{theorem}[Existence of $L^1_+ \cap L^\infty$-solutions $f^\e$]\label{thm_ext} Let $T>0$. Suppose that $f^\e_0$ satisfies
\[
f^\e_0 \in L^1_+ \cap L^\infty(\om \times \R^d) \quad \mbox{and} \quad |v|^2 f^\e_0 \in L^1(\om \times \R^d).
\]
Then there exists a weak solution to \eqref{weak_eq} in the sense of Definition \ref{def_weak} satisfying 
\begin{align}\label{free_ineq}
\begin{aligned}
&\frac{1}{2}\intor |v|^2f^\e\,dxdv+\frac{1}{2\varepsilon}\int_0^t\intor|v|^2(f^\e-M[f^\e])\,dxdvds\cr
&\quad+\frac{1}{2}\int_0^t \intor \,\phi^\varepsilon_\alpha(x-y) |u^\e(x)-u^\e(y)|^2\rho^\varepsilon (x)\rho^\varepsilon (y)\,dxdyds
\le \frac{1}{2}\intor |v|^2f_0^\e\,dxdv.
\end{aligned}
\end{align}
\end{theorem}

We now present our result on the quantified hydrodynamic limit from the BGK-alignment model \eqref{main_eq} to the isentropic EAS \eqref{main_pE}.

\begin{theorem}[Quantified hydrodynamic limit]\label{thm:hydrolimit} Let $T>0$ and $f^\e$ be a weak solution to the BGK-alignment model \eqref{main_eq} obtained in Theorem \ref{thm_ext}. Let $(\rho,u)$ be the unique local-in-time classical solution to the isentropic EAS \eqref{main_pE}  subject to the initial data $(\rho_0,u_0)$ satisfying $\rho \in \mc([0,T]; L^1_+ \cap L^\infty(\om) )$ and $u \in L^\infty(0,T; L^2\cap W^{1,\infty}(\om))$. Suppose that the initial data satisfy the well-prepared conditions:
\begin{itemize}
	\item[{\rm \bf (H1)}]
	\[
	\intor \frac{|v|^2}{2}f^\e_0\,dxdv-\into\rho_0^\e\left(\frac{(\rho_0^\e)^{\gamma-1}}{\gamma-1}+\frac{|u_0^\e|^2}{2}\right)\,d x  \leq C\sqrt\e.
	\]
	\item[{\rm \bf (H2)}]
	\[
	\into  \me(U^\e_0|U_0)\,dx\leq C\sqrt\e.
	\]
\end{itemize}
Moreover, depending on the spatial domain, we assume
\begin{itemize}
\item Periodic domain case $(\om = \T^d)$: $d \geq 1$ and $\alpha \in (0, \min\{2, 1+\frac d2 \})$.

\item Whole space case $(\om = \R^d)$: $d \geq 2$ and $\alpha \in (1,2)$.
\end{itemize}
 Then we have
	\begin{align*}
		%\begin{aligned}\label{convergence}
		&\rho^\e \to\rho \quad \mbox{a.e.} \quad and \quad in \ L^\infty(0,T;L^p(\om))\quad \mbox{with $p \in (1,\gamma)$},\\
		&\rho^\e u^\e \to\rho u \quad \mbox{a.e.} \quad and \quad in \ L^\infty(0,T;L^q(\om)) \quad \mbox{with $q \in [1, \frac{d+2}{d+1})$}, \quad \mbox{and}\cr
		 &\intr v \otimes v f^\e\,dv \to \rho u\otimes u + \rho^\gamma \mathbb{I}_d \quad \mbox{a.e.} \quad and \quad in \  L^1(\om \times (0,T))
		%\end{aligned}
	\end{align*}
	as $\e\to0$. In fact,  there exists a positive constant $C$, which is independent of $\e$ such that 
	\begin{align*}
		%\begin{aligned}\label{rel-ent-est}
			\sup_{t \in [0,T] }\into \left(\rho^\e|u^\e-u|^2 + \mh(\rho^\e| \rho)\right) dx  \le C\e^\lambda,
		%\end{aligned}
	\end{align*}
	where $\lambda:=\min\left\{\frac{1}{2}, \frac\beta2, \frac{\alpha \beta}{4}, \frac{\alpha\beta}{2(\alpha+2)}\right\}$. 
 
\end{theorem}

\subsection{Remarks}\label{sec_rem}

We give several remarks on Theorem \ref{thm:hydrolimit}.

\begin{enumerate}[(i)]

\item Noticing from \eqref{phi_per}, we assume
\[
\intt |x|^r \phi_\alpha(x)\,dx < \infty \quad \iff \quad \alpha -r < d.
\]
\item Since $\alpha \in (0,2)$ and $\e \in (0,1)$, if we choose $\beta = \frac2\alpha$, then 
\[
\sup_{t \in [0,T] }\into \left(\rho^\e|u^\e-u|^2 + \mh(\rho^\e| \rho)\right) dx \leq C \e^{\frac1{\alpha+2}}.
\]

\item  When $\om =\T^d$, by monotonicity of $L^p$-norm, we get
	\begin{align*}
		&\rho^\e \to\rho \quad \mbox{a.e.} \quad and \quad in \ L^\infty(0,T;L^p(\om))\quad \mbox{with $p \in [1,\gamma)$} \quad \mbox{and}\\
		&\rho^\e u^\e \to\rho u \quad \mbox{a.e.} \quad and \quad in \ L^\infty(0,T;L^q(\om)) \quad \mbox{with $q \in [1, \frac{d+2}{d+1})$}
	\end{align*}
	as $\e\to0$.
	
\item  We summarize the current state of the hydrodynamic limit for the BGK-alignment model \eqref{main_eq} toward the isentropic EAS \eqref{main_pE}.
\begin{center}
\begin{tabular}{ |c |c|c| } 
 \hline
  & $d = 1$ & $d \geq 2$ \\
 \hline 
 $\om = \T^d$  & $\alpha \in (0,\frac32)$ & $\alpha \in (0,2)$ \\ 
 \hline 
  $\om = \R^d$ &   & $\alpha \in (1,2)$ \\ 
 \hline
\end{tabular}
\end{center}
Note that for $d\geq 2$, $\alpha \in (0,2)$, and $\tilde \alpha \in (1,2)$, if we consider a communication weight function $\phi \in \mc^1(0,\infty)$ satisfying
\[
\phi(r) = \left\{ \begin{array}{ll}
\displaystyle \frac1{r^\alpha}  & \textrm{if $r \in (0,R)$}\\[3mm]
\displaystyle \frac1{r^{\tilde\alpha}} & \textrm{if $r  > 2R$}
  \end{array} \right.
\]
for some $R>0$, then our main theorem can be directly applied to the above case.

At this moment, we were not able to resolve the case $d=1$ and $\om = \R$. The main difficulty arises from the lower bound estimate on the relative entropy in the one-dimensional case, see Lemma \ref{lem_lower_h} and Remark \ref{rmk_lower_h} for more detailed discussions. 

\item Currently, we were unable to cover the pressure law $p(\rho) = \rho^\gamma$ with $\gamma \in (1,1+\frac2d)$. One could consider 
\bq\label{eq_pos}
\pa_t f^\e + v\cdot\nabla_x f^\e  + \nabla_v \cdot \lt(F_{\phi^\e_\alpha}[f^\e]f^\e\rt)  = \frac{1}{\e}Q[f^\e], 
\eq
where $Q[f] = M[f] - f$ with
\[
M[f] = c_{\gamma,d}\left(\frac{2\gamma}{\gamma-1}\rho_f^{\gamma-1}-|v-u_f|^2\right)^{n/2}_+, \quad \gamma \in (1, 1 + \frac2d)
\]
and pass to the limit $\e \to 0$ to have the isentropic EAS with $\gamma \in (1, 1+ \frac2d)$. Following \cite{Bou99}, in this case, the kinetic entropy for the above kinetic equation would be defined by
\[
H(f) = \frac{|v|^2}2 f + \frac1{2c_{\gamma,d}^{2/n}} \frac{f^{1 + 2/n}}{1 + 2/n},
\]
where $c_{\gamma,d}$ and $n$ are given as in \eqref{para}. Then we get from \cite{Bou99} that
\[
\intr H(M[f])\,dv \leq \intr H(f)\,dv,
\]
and this gives some dissipation from the BGK operator $Q$. However, in this approach, we could not find an appropriate way of having non-increasing of $\intor H(f^\e)\,dxdv$ due to the present of the non-local alignment force $F[f]$ even with regular communication weights. 

%some bound on $\|f^\e\|_{L^p}$ with $p=1 + 2/n$ uniformly in $\e$ due to the singular communication weights. 
%On the other hand, when the communication weight is bound, the isentropic EAS with $\gamma \in [1,1 + \frac2d]$ can be rigorously derived. We give some details of that in Section \ref{sec_hydro2}. 
\end{enumerate}

\subsection{Organization of the paper}

The rest of the paper is organized as follows. In Section \ref{sec_pre}, we provide a uniform-in-$\e$ bound estimate on the kinetic energy with appropriate dissipations. We also present lower bound estimates on  $\mh(\bar\rho | \rho)$. Those two estimates will be significantly used in the relative entropy estimates later. In order to emphasize our main result on the derivation of \eqref{main_pE} from \eqref{main_eq}, we give the details of the proof of Theorem \ref{thm:hydrolimit} in Section \ref{sec_hydro}. Finally, in Section \ref{sec_weak}, we construct the global-in-time weak solutions to the BGK-alignment model \eqref{main_eq} satisfying the required entropy inequality \eqref{free_ineq}.

%%%%%%%%%%%%%%%%%%%%%%%%%%%%%%%%%%%%%%%%%%%%%%%%%%%%%%%%%%%%%%
%
%
%.  
%.  
%
%
%%%%%%%%%%%%%%%%%%%%%%%%%%%%%%%%%%%%%%%%%%%%%%%%%%%%%%%%%%%%%%%%%%%%
%\subsection{Isentropic EAS with $p(\rho)=\rho^\gamma$ with $\gamma \in [1,1 + \frac2d)$}

%%%%%%%%%%%%%%%%%%%%%%%%%%%%%%%%%%%%%%%%%%%%%%%%%%%%%%%%%%%%%%
%
%
%.  
 \section{Preliminaries}\label{sec_pre}
%
%
%%%%%%%%%%%%%%%%%%%%%%%%%%%%%%%%%%%%%%%%%%%%%%%%%%%%%%%%%%%%%%%%%%%%

\subsection{Uniform kinetic energy estimate on $f^\e$}\label{pre_uni}

In this subsection, we provide the uniform-in-$\e$ bound estimate on the kinetic energy for \eqref{main_eq}. For this, we consider the kinetic energy:
\[
H(f)=\frac{|v|^2}{2}f.  
\]
Multiplying \eqref{main_eq} by $|v|^2$ and integrating over $x$ and $v$, we get
\begin{align}\label{kinetic energy bound}\begin{split}
\frac{d}{dt}\iint_{\Omega\times\mathbb{R}^d}\frac{|v|^2}{2} f^\e\,dxdv     &=\iint_{\Omega\times\mathbb{R}^d} v \cdot \lt(F_{\phi^\varepsilon_\alpha}[f^\e]f^\e\rt)\,dxdv+ \frac{1}{\e} \iint_{\Omega\times\mathbb{R}^d}\frac{|v|^2}{2}(M[f^\e]-f^\e)\,dxdv.
\end{split}\end{align}
Here we use the same argument as in \cite{CKpre22} to get 
\begin{align*}
\intor \left(v\cdot F_{\phi^\varepsilon_\alpha}[f^\varepsilon] \right)f^\varepsilon\,dxdv&\le-\frac{1}{2}\intoo \phi^\varepsilon_\alpha(x-y)\left|u^\e(x)-u^\e(y)\right|^2 \rho^\varepsilon(x)\rho^\varepsilon(y)\,dxdy.
\end{align*}
This combined with \eqref{kinetic energy bound} gives 
\begin{align*} 
\begin{aligned}
&\frac{1}{2}\intor |v|^2f^\e\,dxdv+\frac{1}{\varepsilon}\int_0^t\intor|v|^2(f^\e-M[f^\e])\,dxdvds\cr
&\quad+\frac{1}{2}\int_0^t \intoo \,\phi^\varepsilon_\alpha(x-y) |u^\e(x)-u^\e(y)|^2\rho^\varepsilon (x)\rho^\varepsilon (y)\,dxdyds
\leq \frac{1}{2}\intor |v|^2f_0^\e\,dxdv.
\end{aligned}
\end{align*}
Note from \cite{Bou99} that for any $f$ satisfying 
$$
\intr f+H(f )\,dv < \infty,
$$
the minimization principle holds true:
\begin{equation}\label{mini} 
\intr  H(M[f] )\,dv\le \intr  H(f )\,dv.
\end{equation} 
This together with \eqref{kinetic energy bound} gives the bound of kinetic entropy:
 \begin{align}\label{energy bound2}
 \begin{aligned}
 &\frac{1}{2}\intor |v|^2f^\e\,dxdv+\frac{1}{2}\int_0^t \intoo \,\phi^\varepsilon_\alpha(x-y) |u^\e(x)-u^\e(y)|^2\rho^\varepsilon (x)\rho^\varepsilon (y)\,dxdyds\cr
 &\quad
 \le \frac{1}{2}\intor |v|^2f_0^\e\,dxdv.
 \end{aligned}
 \end{align}

 %%%%%%%%%%%%%%%%%%%%%%%%%%%%%%%%%%%%%%%%%%%%%%%%%%%%%%%%%%%%%%
%
%
%.  
%.  
%
%
%%%%%%%%%%%%%%%%%%%%%%%%%%%%%%%%%%%%%%%%%%%%%%%%%%%%%%%%%%%%%%%%%%%%

\subsection{Lower bound estimate on $\mh(\bar\rho | \rho)$}

\begin{lemma}\label{lem_lower_h} Let $\mh(\bar\rho | \rho)$ be given as \eqref{def_h}. Suppose that 
\[
\into \bar\rho\,dx = \into \rho\,dx = 1.
\]
Then there exists $C>0$ independent of $\bar\rho$ and $\rho$ such that 
\[
\into \mh(\bar\rho| \rho)\,dx  \geq C\left\{ \begin{array}{ll}
\|\bar\rho - \rho\|_{L^{\frac{2}{3-\gamma}}}^{2} & \textrm{if $d \geq 2$}\\[4mm]
\|\bar\rho - \rho\|_{L^2(\rho)}^2 & \textrm{if $d=1$}
  \end{array} \right..
\]
In particular, if $\om = \T^d$, $d=1$, and $\rho \geq \rho_m > 0$ for some constant $\rho_m$, then 
\[
\into \mh(\bar\rho| \rho)\,dx \geq  C\rho_m\|\bar\rho - \rho\|_{L^2}^2.
\]
\end{lemma}
\begin{proof}
Note that if $d \geq 2$, then $\gamma = 1 + \frac2d \leq 2$, and in this case, by Taylor's theorem, we readily see
\[%\begin{equation}\label{rel_ent_l2}
\mh(\bar\rho| \rho) \geq \frac{\gamma}{2} \min\lt\{\frac{1}{\bar\rho^{2-\gamma}}, \frac{1}{\rho^{2-\gamma}} \rt\}(\rho - \bar\rho)^2,
\]%\end{equation}
and  
\begin{align*}
\into \left|\rho^\e-\rho\right|^{\frac{2}{3-\gamma}}\,dx&=\into \min\left\{\frac{1}{\bar\rho^{2-\gamma}},\frac{1}{\rho^{2-\gamma}} \right\}^{\frac{1}{3-\gamma}}\max\left\{\bar\rho^{2-\gamma},\rho^{2-\gamma} \right\}^{\frac{1}{3-\gamma}}\left|\rho^\e-\rho\right|^{\frac{2}{3-\gamma}}\,dx\cr
&\le \left(\into \min\lt\{\frac{1}{\bar\rho^{2-\gamma}}, \frac{1}{\rho^{2-\gamma}} \rt\}\left|\rho - \bar\rho\right|^2\,dx \right)^{\frac{1}{3-\gamma}} \left(\into \max\left\{\bar\rho,\rho \right\}\,dx \right)^{\frac{2-\gamma}{3-\gamma}}. 
\end{align*}
This yields
\begin{equation*}%\label{3-gamma}
\|\bar\rho - \rho\|_{L^{\frac{2}{3-\gamma}}}^{2} \leq C\into \mh(\bar\rho| \rho)\,dx
\end{equation*}
for some $C>0$.
 
On the other hand, if $d=1$, then $\gamma = 3$, and by Taylor's theorem we find
\[
\mh(\bar\rho | \rho) = \frac{3}{2}\int_0^1 (1-\theta)(\rho + \theta(\bar\rho - \rho))(\bar\rho - \rho)^2\,d\theta.
\]
In particular, this implies
\[
\mh(\bar\rho | \rho) \geq \frac32 \rho (\bar\rho - \rho)^2 \int_0^1 (1-\theta)^2\,d\theta = \frac12 \rho(\bar\rho - \rho)^2.
\]
This completes the proof.
\end{proof}

\begin{remark}\label{rmk_lower_h} For the one-dimensional case, we obtain
\[
\|\bar\rho - \rho\|_{L^3}^2 \leq C\lt(\into \mh(\bar\rho| \rho)\,dx  \rt)^{\frac23}
\]
for some $C>0$. Indeed, if $\bar\rho \geq \rho$, then
\[
\int_0^1 (1-\theta)(\rho + \theta(\bar\rho - \rho))(\bar\rho - \rho)^2\,d\theta \geq (\bar\rho - \rho)^3 \int_0^1 (1-\theta)\theta\,d\theta = \frac{(\bar\rho - \rho)^3 }6.
\]
On the other hand, if $\bar\rho \leq \rho$, then
\[
\int_0^1 (1-\theta)(\rho + \theta(\bar\rho - \rho))(\bar\rho - \rho)^2\,d\theta \geq  \rho (\bar\rho - \rho)^2 \int_0^1 (1-\theta)^2\,d\theta = \frac{\rho(\bar\rho - \rho)^2}3.
\]
Thus we obtain
\begin{align*}
\into |\bar\rho - \rho|^3\,dx &= \int_{\bar\rho \geq \rho} |\bar\rho - \rho|^2 (\bar\rho - \rho)\,dx + \int_{\bar\rho \leq \rho} |\bar\rho - \rho|^2 (\rho - \bar\rho)\,dx\cr
&\leq \int_{\bar\rho \geq \rho} |\bar\rho - \rho|^2 (\bar\rho - \rho)\,dx + \int_{\bar\rho \leq \rho} |\bar\rho - \rho|^2 \rho \,dx \cr
&\leq C\into \mh(\bar\rho| \rho)\,dx
\end{align*}
for some $C>0$.

It is worth noticing that in the estimates of relative entropy, it is required to obtain
\bq\label{hope}
\|\bar\rho - \rho\|_{L^p}^2 \leq C\into \mh(\bar\rho| \rho)\,dx 
\eq
for some $p \in [1,\gamma]$. As stated in Lemma \ref{lem_lower_h}, we have the above estimate with $p=2$ when $\om = \T^d$ by assuming that $\rho \geq \rho_m > 0$. However, in the case $\om =\R$, we cannot assume that $\rho$ has a positive bound due to $\rho \in L^1(\R)$. This causes a great difficulty in having \eqref{hope} and this is the main reason why we could not handle the case $\om = \R$.

\end{remark}

%%%%%%%%%%%%%%%%%%%%%%%%%%%%%%%%%%%%%%%%%%%%%%%%%%%%%%%%%%%%%%
%
%
%.  
%.  
%
%
%%%%%%%%%%%%%%%%%%%%%%%%%%%%%%%%%%%%%%%%%%%%%%%%%%%%%%%%%%%%%%%%%%%%

\section{Hydrodynamic limit from the BGK-alignment model to isentropic EAS}\label{sec_hydro}

\subsection{Relative entropy estimate}
We now provide the estimate of the relative entropy functional $\mathcal{E}(U^\e|U)$.
\begin{lemma}\label{lem_rel}
The relative entropy $\me(U^\e|U)$ defined in \eqref{def_rel} satisfies the following estimate:
\begin{align}
\begin{aligned}\label{D-5}
&\into \me(U^\e|U)\,dx \cr
&\quad \leq \into \me(U_0^\e|U_0)\,dx + \into E(U^\e) - E(U^\e_0)\,dx  \cr
&\qquad +\frac12\int_0^t \intoo \rho^\e(x) \rho^\e(y)\phi_\alpha^\e(x-y)|u^\e(x) - u^\e(y)|^2\,dxdyds \cr
&\qquad - \int_0^t \into \nabla (DE(U)):A(U^\e|U)\,dxds  - \int_0^t \into DE(U)\cdot (\pa_t U^\e + \nabla \cdot A(U^\e)- F^\e(U^\e))\,dxds \cr
&\qquad   + \int_0^t \intoo \rho^\e(x) (\rho(y) - \rho^\e(y))\phi_\alpha(x-y) (u^\e(x) - u(x))\cdot (u(y) - u(x))\,dxdyds\cr
&\qquad  + \frac{1}{2}\int_0^t\intoo (\phi^\e_\alpha(x-y)-\phi_\alpha(x-y)) (u(x)-u(y)) \cdot (u^\e(y) - u^\e(x)) \rho^\e(x) \rho^\e(y)\,dxdyds,
\end{aligned}
\end{align}
where   $A:B := \sum_{i=1}^m \sum_{j=1}^n a_{ij} b_{ij}$ for $A, B \in \R^{m\times n}$ and $A(U^\e|U)$ is the relative flux functional given by
\[
A(U^\e|U) := A(U^\e) - A(U) - DA(U)(U^\e-U).
\]
\end{lemma}
\begin{proof}
It follows from the definition of relative entropy functional \eqref{def_rel}, \cite[Lemma 4.1]{KMT15}, and \cite[Proof of Proposition 4.2]{KMT15}  that
\[
\begin{aligned}
\frac{d}{dt}\into \me(U^\e|U)\,dx & = \into \partial_t E(U^\e)\,dx - \into DE(U)\cdot (\pa_t U^\e + \nabla \cdot A(U^\e)- F^\e(U^\e))\,dx \cr
&\quad  - \into \nabla (DE(U)):A(U^\e|U)\,dx \underbrace{-\into D^2E(U)F(U)(U^\e-U) + DE(U)\cdot F^\e(U^\e)\,dx}_{=:I}.
\end{aligned}
\]
On the other hand, straightforward computation gives
$$\begin{aligned}
&D^2E(U)F(U)(U^\e-U) \cr
&\quad = \rho^\e(x) (u^\e(x) - u(x))\cdot\lt( \into \phi_\alpha(x-y)(u(y) - u(x)) \rho(y)\,dy \rt) 
\end{aligned}$$
and
\[
DE(U)\cdot F^\e(U^\e) = \rho^\e u \cdot \lt(  \into \phi^\e_\alpha(x-y)(u^\e(y) - u^\e(x)) \rho^\e(y)\,dy  \rt).
\]
We then use the same argument as in \cite{CKpre22} to have
\begin{align*}
I &= \frac12\intoo \rho^\e(x) \rho^\e(y)\phi_\alpha^\e(x-y)|u^\e(x) -  u^\e(y)|^2\,dxdy\cr
&\quad -\frac12\intoo \rho^\e(x) \rho^\e(y)\phi_\alpha^\e(x-y)|(u^\e(x) - u(x)) - (u^\e(y) - u(y))|^2 dxdy\cr
&\quad - \frac{1}{2}\int_0^t\intoo (\phi_\alpha(x-y)-\phi^\e_\alpha(x-y))|u(x)-u(y)|^2\rho^\e(x)\rho^\e(y)\,dxdyds\cr
&\quad + \intoo\rho^\e(x) (\rho(y) - \rho^\e(y))\phi_\alpha(x-y) (u^\e(x) - u(x))\cdot (u(y) - u(x))\,dxdy \cr
&\quad + \frac{1}{2}\intoo (\phi^\e_\alpha(x-y)-\phi_\alpha(x-y)) (u(x)-u(y)) \cdot (u^\e(y) - u^\e(x)) \rho^\e(x)\rho^\e(y)\,dxdy\cr
&\leq \frac12\intoo \rho^\e(x) \rho^\e(y)\phi_\alpha^\e(x-y)|u^\e(x) -  u^\e(y)|^2\,dxdy\cr
&\quad + \intoo\rho^\e(x) (\rho(y) - \rho^\e(y))\phi_\alpha(x-y) (u^\e(x) - u(x))\cdot (u(y) - u(x))\,dxdy \cr
&\quad + \frac{1}{2}\intoo (\phi^\e_\alpha(x-y)-\phi_\alpha(x-y)) (u(x)-u(y)) \cdot (u^\e(y) - u^\e(x)) \rho^\e(x)\rho^\e(y)\,dxdy,
\end{align*}
due to $\phi_\alpha \geq \phi_\alpha^\e \geq 0$. This together with integrating the resulting inequality with respect to time concludes the desired result.
\end{proof}

\subsection{Proof of Theorem \ref{thm:hydrolimit}}
We first recall from \cite{CKpre22} an auxiliary lemma which gives the error bound on $\phi_\alpha - \phi^\e_\alpha$.
\begin{lemma}\label{phi-est}
	Let $\alpha \in (0,2)$. Then
	\[
	\phi_\alpha(x)-\phi^\e_\alpha(x)\le \frac{\alpha \e^\beta}{2}\frac{\phi^\e_\alpha(x)}{|x|^2} \quad \mbox{for} \quad x \in \R^d \setminus \{0\}.
	\]
\end{lemma}

We next estimate the each term on the right hand side of \eqref{D-5}. Let us set each term $I_i, i=1,\dots,7$, i.e.,
\[
\into \me(U^\e|U)\,dx \leq \sum_{i=1}^7 I_{i}.
\]
 
\bigskip

\noindent {\bf Estimate of $I_1$}:  It follows from the well-prepared initial data assumption {\bf (H2) }that
\[
I_1 \leq C\sqrt\e.
\]

\bigskip

\noindent {\bf Estimate of $I_2 +I_3$}: Note from  \cite{Bou99} that for any $f$ satisfying
$$
\intr f +H(f )\,dv < \infty,
$$
a compatibility between the entropy $E$ and the kinetic entropy $H$ is satisfied as
\[
\intr H(M[f] )\,dv=E(U).
\]
This, combined with \eqref{mini} yields
\[
\into E(U^\e)\,d x = \intor  H(M[f^\e])\,dxdv \leq \intor  H(f^\e)\,dxdv=\frac12\iint_{\Omega\times\mathbb{R}^d}|v|^2 f^\e\,dxdv.
\]
Thus, we have  
$$\begin{aligned}
&I_2 + I_3 \cr
&\ = \into E(U^\e)\,dx - \intor  H(f^\e)\,dxdv \cr
&\quad + \intor  H(f^\e)\,dxdv +\frac12\int_0^t \intoo \rho^\e(x) \rho^\e(y)\phi_\alpha^\e(x-y)|u^\e(x) - u^\e(y)|^2\,dxdyds- \intor  H(f^\e_0)\,dxdv\cr
&\quad  + \intor  H(f^\e_0)\,dxdv - \into E(U^\e_0)\,dx \cr
&\ \leq  C\sqrt\e,
\end{aligned}$$
where we used the estimate of kinetic entropy \eqref{energy bound2} and the well-prepared initial data assumption {\bf (H1)} in the last inequality.

\bigskip

\noindent {\bf Estimate of $I_4$}: It follows from \cite[Lemma 4.3]{KMT15} that
\[
A(U^\e|U) = \begin{pmatrix}
    0   \\[4mm]
        \rho^\e(u^\e - u) \otimes (u^\e - u)  + (\gamma-1) \mh(\rho^\e|\rho) \mathbb{I}_d
    \end{pmatrix}.
\]
This yields
$$\begin{aligned}
|I_4| &= \lt|\int_0^t \into \nabla u : (\rho^\e(u^\e - u) \otimes (u^\e - u) + \frac2d \mh(\rho^\e|\rho) \mathbb{I}_d)\,dxds \rt| \cr
&\leq  \|\nabla u\|_{L^\infty}\int_0^t \into \me(U^\e|U)\,dxds.
\end{aligned}$$

\bigskip

\noindent {\bf Estimate of $I_5$}:  We first record that the Maxwellian distribution satisfies (see \cite{Bou99, CHpre})
\begin{align*}%\label{cancellation}
\begin{split}
\intr M(f^\e)\,dv&=\rho^\e, \quad \intr vM(f^\e)\,dv=\rho^\e u^\e, \quad  \mbox{and} \quad \intr v \otimes v M(f^\e)\,dv= \rho^\e u^\e \otimes u^\e + C_d(\rho^\e)^\gamma \mathbb{I}_d,
\end{split}
\end{align*}
where $C_d$ is appeared in  \eqref{c_d}. By using the above, we find
\[
 \partial_t U^\e+\pa_x A(U^\e)-F^\e(U^\e)=\begin{pmatrix} 0 \\ \nabla \cdot \lt( \displaystyle \intr (u^\e \otimes u^\e - v \otimes v)f^\e\,dv + C_d(\rho^\e)^\gamma \mathbb{I}_d \rt)  \end{pmatrix},
\]
which gives
\begin{align*}
I_{5 }&= \int_0^t\into (\nabla u) :\lt( \displaystyle \intr (u^\e \otimes u^\e - v \otimes v)f^\e\,dv + C_d(\rho^\e)^\gamma \mathbb{I}_d \rt)\,dxds\cr
&\leq C\int_0^t\into \left|\intr v \otimes v\left( M[f^\e]-f^\e\right) dv \right| dxds.
\end{align*}
 
\begin{lemma}\cite[Proposition 4.1]{BV05}\label{lem_gd}
Let $f^\e$ be a solution to the BGK equation with initial value $f^\e_0$ bounded in $L^1(\Omega\times\mathbb{R}^{d})$ verifying
\begin{equation*}
\intor |v|^2f_0^\e (x,v)\,dxdv \le C^0<\infty,
\end{equation*}
and with $\gamma=1 + \frac2d$. Then there exists $C>0$ such that for every $\e<1$, we have
$$
\int_0^T\into \left|\intr v\otimes v \left(M[f^\e]- f^\e\right) dv\right|\,dxdt \le C\sqrt{\e}.
$$
\end{lemma}
Thus we get
\[
I_{5 }\le C\sqrt{\e}.
\]

\bigskip

\noindent {\bf Estimate of $I_6$}:  We estimate $I_6$ by dividing two cases $\om = \T^d$ and $\om = \R^d$.

\medskip

$\bullet$ {\bf Case A} ($\om = \T^d$): We deal with the cases $d = 1$ and $d \geq 2$, separately.

\medskip

$\bullet$ {\bf Case A.1} ($d=1$): In this case, $\gamma = 3$ and $\rho$ has a strict positive lower bound $\rho_m > 0$. Since 
\[
\int_\T(\rho^\e - \rho)^2\,dx \leq \int_\T \frac{2}{\rho}\mh(\rho^\e | \rho)\,dx \leq \frac2{\rho_m}\int_\T\mh(\rho^\e | \rho)\,dx,
\]
we get
\bq\label{l2_bdd}
\|\rho^\e - \rho\|_{L^2}^2 \leq C\int_\T\mh(\rho^\e | \rho)\,dx
\eq
for some $C>0$ independent of $\e>0$.

Then, for $\alpha \in (0, \frac32)$, we estimate
\begin{align*}%\label{alpha}
\begin{split}
I_6 &\leq  {\rm Lip}(u)\int_0^t \iint_{\T \times \T}\rho^\e(x) |x-y|\phi_\alpha(x-y) |\rho(y) - \rho^\e(y)| |u^\e(x) - u(x)|\,dxdyds\\
&\le  {\rm Lip}(u)\int_0^t\left(\int_\T\rho^\e(x)\left|u^\e(x)-u(x)\right|^2\,dx\right)^{\frac12}\cr
& \hspace{2cm}  \times  \int_\T \left(\int_\T\rho^\e(x) |x-y|^2\phi_{2\alpha}(x-y)\,dx\right)^{\frac12} |\rho(y)-\rho^\e(y)|\,dyds
\end{split}
\end{align*}

On the other hand,
\begin{align*}
&\int_\T \left(\int_\T\rho^\e(x) |x-y|^2\phi_{2\alpha}(x-y)\,dx\right)^{\frac12} |\rho(y)-\rho^\e(y)|\,dy\cr
&\quad \leq \lt(\iint_{\T\times \T}\rho^\e(x) |x-y|^2\phi_{2\alpha}(x-y)\,dxdy\rt)^\frac12 \|\rho - \rho^\e\|_{L^2}\cr
&\quad \leq C\lt(\int_\T\mh(\rho^\e | \rho)\,dx\rt)^\frac12,
\end{align*}
where we used \eqref{l2_bdd} and
\[
\int_\T|x-y|^2\phi_{2\alpha}(x-y)\,dy < \infty
\]
due to $\alpha < \frac32$. This implies
\[
I_6 \leq C\int_0^t\left(\int_\T\rho^\e(x)\left|u^\e(x)-u(x)\right|^2\,dx\right)^{\frac12}\lt(\int_\T\mh(\rho^\e | \rho)\,dx\rt)^\frac12 ds \leq C\int_0^t \int_\T \me(U^\e|U)\,dx ds,
\]
where $C>0$ is independent of $\e>0$.

\medskip

$\bullet$ {\bf Case A.2} ($d\geq 2$): In this case, we notice that $\gamma \leq 2$ and 
\[
\|\rho^\e - \rho\|_{L^{\frac{2}{3-\gamma}}}^{2} \leq C\into \mh(\rho^\e| \rho)\,dx
\]
for some $C>0$ independent of $\e>0$. We use a similar argument as in {\bf Case A.1} to obtain
\begin{align*}%\label{alpha}
\begin{split}
I_6 &\le  {\rm Lip}(u)\int_0^t\left(\intt \rho^\e(x)\left|u^\e(x)-u(x)\right|^2\,dx\right)^{\frac12}\cr
& \hspace{2cm}  \times  \intt \left(\intt \rho^\e(x) |x-y|^2\phi_{2\alpha}(x-y)\,dx\right)^{\frac12} |\rho(y)-\rho^\e(y)|\,dyds
\end{split}
\end{align*}
Applying H\"{o}lder's inequality and Minkowski's inequality, we get 
\begin{align*}%\label{alpha1}
\begin{split}
&\intt \left(\intt\rho^\e(x) |x-y|^2\phi_{2\alpha}(x-y)\,dx\right)^{\frac12} |\rho(y)-\rho^\e(y)|\,dy\cr
&\quad\le \left\|\rho(y)-\rho^\e(y)\right\|_{L^{\frac{2}{3-\gamma}}}\left(\intt \left(\intt\rho^\e(x) |x-y|^2\phi_{2\alpha}(x-y)\,dx\right)^{\frac{1}{\gamma-1}} \,dy\right)^{\frac{\gamma-1}{2}}\cr
&\quad\le \left\|\rho(y)-\rho^\e(y)\right\|_{L^{\frac{2}{3-\gamma}}}\left(\intt \rho^\e(x) \left(\intt |x-y |^{\frac{2}{\gamma-1}}\phi_{\frac{2\alpha}{\gamma-1}}(x-y)   \,dy\right)^{\gamma-1} \,dx\right)^{\frac12}.
\end{split}
\end{align*}
Note that $\frac{2}{\gamma-1} = d$ and 
\[
\intt |x-y |^{\frac{2}{\gamma-1}}\phi_{\frac{2\alpha}{\gamma-1}}(x-y)   \,dy < \infty
\]
since 
\[
\frac{2}{\gamma-1}(\alpha - 1) < d \quad \iff \quad \alpha < 2.
\]
Hence,
\[
I_6 \leq C\int_0^t \intt \me(U^\e|U)\,dx ds,
\]
where $C>0$ is independent of $\e>0$.

\medskip

$\bullet$ {\bf Case B} ($\om = \R^d$): In this case, we only consider $d\geq 2$ and $\alpha \in (1,2)$, and this subsequently implies $\gamma \leq 2$. Analogously as the above, we deduce
\begin{align*}%\label{alpha}
\begin{split}
I_6 &\le   \int_0^t\left(\intr \rho^\e(x)\left|u^\e(x)-u(x)\right|^2\,dx\right)^{\frac12}\cr
& \hspace{1cm}  \times \left\|\rho(y)-\rho^\e(y)\right\|_{L^{\frac{2}{3-\gamma}}}\left(\intr \rho^\e(x) \left(\intr |u(x)-u(y) |^{\frac{2}{\gamma-1}}\phi_{\frac{2\alpha}{\gamma-1}}(x-y)   \,dy\right)^{\gamma-1} \,dx\right)^{\frac12}ds
\end{split}
\end{align*}
We then estimate
\begin{align*}
 &\intr |u(x)-u(y) |^{\frac{2}{\gamma-1}}\phi_{\frac{2\alpha}{\gamma-1}}(x-y)   \,dy\cr 
 &\quad= \int_{|x-y|<R} |u(x)-u(y) |^{\frac{2}{\gamma-1}}\phi_{\frac{2\alpha}{\gamma-1}}(x-y) \,dy+\int_{|x-y|\ge R} |u(x)-u(y) |^{\frac{2}{\gamma-1}}\phi_{\frac{2\alpha}{\gamma-1}}(x-y) \,dy \cr
 &\quad\le {\rm Lip}(u)\int_{|x-y|<R} |x-y|^{\frac{2}{\gamma-1}}\phi_{\frac{2\alpha}{\gamma-1}}(x-y) \,dy+ \|u\|_{L^\infty}^{\frac{2}{\gamma-1}}\int_{|x-y|\ge R}\phi_{\frac{2\alpha}{\gamma-1}}(x-y)\,dy,
\end{align*}
where the first integral on the right hand is finite due to $\alpha < 2$, and the second integral is also finite since 
\[
\frac{2}{\gamma-1}\alpha > d \quad \iff \quad \alpha > 1.
\]
Thus, we have 
\[
I_6 \le C\int_0^t \into \me(U^\e|U)\,dx ds. 
\]

\bigskip

\noindent {\bf Estimate of $I_7$}: We first estimate $I_7$ as
\begin{align*}
I_7 &\leq \frac{{\rm Lip}(u)}{2}\int_0^t\intoo (\phi_\alpha(x-y)-\phi^\e_\alpha(x-y))|x-y||u^\e(y) - u^\e(x)| \rho^\e(x) \rho^\e(y)\,dxdyds\cr
&= \frac{{\rm Lip}(u)}{2}\int_0^t (I_{71} + I_{72})\,ds,
\end{align*}
where 
\[
I_{7i} := \iint_{A_i} (\phi_\alpha(x-y)-\phi^\e_\alpha(x-y))|x-y||u^\e(y) - u^\e(x)| \rho^\e(x) \rho^\e(y)\,dxdy, \quad i=1,2,
\]
with $A_1 = \{(x,y) \in \om \times \om : |x-y| > \delta\}$ and $A_2 = \{(x,y) \in \om \times \om : |x-y| \leq \delta\}$. Here $\delta \in (0,1)$ will be chosen appropriately later.

For $I_{71}$, we use Lemma \ref{phi-est} and almost the same argument as in \cite{CKpre22} to get
\begin{align*}
I_{71} 	&\le \frac{\alpha \e^\beta}{2} \iint_{|x-y| > \delta}\rho^\e(x)\rho^\e(y) \frac{\phi^\e_\alpha(x-y)}{|x-y|}|u^\e(x)-u^\e(y)|\,dxdy \\
	&\le C\frac{\e^{\beta}}{\delta}\left( \iint_{|x-y|>\delta}\rho^\e(x)\rho^\e(y)\phi^\e_\alpha(x-y)\,dxdy \right)^{1/2}\\
	&\qquad \times\left( \iint_{|x-y|> \delta }\phi^\e_\alpha(x-y)\rho^\e(x)\rho^\e(y)|u^\e(x)-u^\e(y)|^2\,dxdy \right)^{1/2}\cr
	&\leq C\frac{\e^{\beta}}{\delta^{1 + \frac\alpha2}}\left( \iint_{|x-y|> \delta } \phi^\e_\alpha(x-y)\rho^\e(x)\rho^\e(y)|u^\e(x)-u^\e(y)|^2\,dxdy \right)^{1/2}\
\end{align*}
due to 
\[
 \iint_{|x-y|>\delta} \rho^\e(x)\rho^\e(y)\phi^\e_\alpha(x-y)\,dxdy \le  \iint_{|x-y|>\delta} \rho^\e(x)\rho^\e(y)|x-y|^{-\alpha}\,dxdy \leq \delta^{-\alpha} .
\]
For the estimate of $I_{72}$, we obtain
\begin{align*}
	I_{72}&\le  \iint_{|x-y|<\delta}\rho^\e(x)\rho^\e(y)\phi_\alpha(x-y)|x-y||u^\e(x)-u^\e(x)|\,dxdy \\
	&\le \iint_{|x-y|<\delta}\rho^\e(x)\rho^\e(y) |x-y|^{1-\alpha}|u^\e(x)-u^\e(y)|\,dxdy \\
	&\leq \lt(\iint_{|x-y|<\delta} \frac{\rho^\e(x)\rho^\e(y)}{|x-y|^{2(\alpha-1)}}\,dxdy\rt)^{\frac12}   \lt(\iint_{|x-y|<\delta} \rho^\e(x)\rho^\e(y)|u^\e(x)-u^\e(y)|^2 \,dxdy\rt)^{\frac12}  \cr
	&\le   (\e^\beta+\delta^{2})^{\frac{\alpha}{4}}\lt(\iint_{|x-y|<\delta} \frac{\rho^\e(x)\rho^\e(y)}{|x-y|^{2(\alpha-1)}}\,dxdy\rt)^{\frac12}\cr
	&\hspace{1.5cm} \times \left(  \iint_{|x-y|<\delta}\phi^\e_\alpha(x-y)\rho^\e(x)\rho^\e(y)|u^\e(x)-u^\e(y)|^2\,dxdy\right)^{1/2},
	\end{align*}
		where we used $1\le \phi^\e_\alpha(x-y)(\e^\beta+\delta^2)^{\frac{\alpha}{2}}$ for $|x| < \delta$.
		
We then consider two cases $\alpha \in (0,1]$ and $\alpha \in (1, \min \{ 2, d+\frac12\})$. When $\alpha \in (0,1]$, we get
\[
\lt(\iint_{|x-y|<\delta} \frac{\rho^\e(x)\rho^\e(y)}{|x-y|^{2(\alpha-1)}}\,dxdy\rt)^{\frac12} \leq \delta^{1-\alpha}\lt(\intoo \rho^\e(x)\rho^\e(y) \,dxdy\rt)^{\frac12} = \delta^{1-\alpha}.
\]
Thus,
\[
I_{72} \leq \delta^{1-\alpha}(\e^\beta+\delta^{2})^{\frac{\alpha}{4}}\left(  \iint_{|x-y|<\delta}\phi^\e_\alpha(x-y)\rho^\e(x)\rho^\e(y)|u^\e(x)-u^\e(y)|^2\,dxdy\right)^{1/2}.
\]
In order to handle the case  $\alpha \in (1, \min \{ 2, d+\frac12\})$ , we recall the classical Hardy-Littlewood-Sobolev inequality:
\[
\lt|\intoo  \mu(x)|x-y|^{-\lambda}\nu(y)\,dxdy \rt| \leq C \|\mu\|_{L^p}\|\nu\|_{L^q}
\]
for $\mu \in L^p(\om)$, $\nu \in L^q(\om)$, $1 < p, q <\infty$, $\frac1p + \frac1q + \frac\lambda d = 2$, and $0 < \lambda < d$.

For $d \geq 2$, we find $2(\alpha-1) < d$ due to $\alpha \in (1,2)$. Thus we use the above Hardy-Littlewood-Sobolev inequality with $\lambda = 2(\alpha-1)$ and $p=q$ to obtain
\[
\iint_{|x-y|<\delta} \frac{\rho^\e(x)\rho^\e(y)}{|x-y|^{2(\alpha-1)}}\,dxdy \leq C\|\rho^\e\|_{L^p}^2
\]
for $d \geq 2$. Here $p$ is given by
\[
1 < p = \frac{d}{d+1 - \alpha} < \frac{d}{d-1} \leq  1 + \frac2d =  \gamma.
\]
Since 
\[
\into \rho^\e |u^\e|^2 + dC_d (\rho^\e)^\gamma \,dx = \intor |v|^2 M(f^\e)\,dxdv \leq \intor |v|^2 f^\e\,dxdv \leq \intor |v|^2 f^\e_0\,dxdv,
\]
we get
\[
\|\rho^\e\|_{L^1 \cap L^\gamma} \leq C
\]
for some $C>0$ independent of $\e$, and subsequently by $L^p$ interpolation inequality, 
\[
\iint_{|x-y|<\delta} \frac{\rho^\e(x)\rho^\e(y)}{|x-y|^{2(\alpha-1)}}\,dxdy \leq C
\]
for $d \geq 2$.

When $d=1$, we consider $\alpha \in (0,\frac32)$ and this gives $2(\alpha - 1) < 1$. Then analogously as above, 	
\[
\iint_{|x-y|<\delta} \frac{\rho^\e(x)\rho^\e(y)}{|x-y|^{2(\alpha-1)}}\,dxdy \leq C\|\rho^\e\|_{L^p}^2,
\]
where 
\[
1 < p = \frac1{2-\alpha} < 2 < 3 = \gamma.
\]
Thus we also have the uniform-in-$\e$ bound on $\|\rho^\e\|_{L^p}$ in the one-dimensional case.

Thus, for $\alpha \in (1, \min \{ 2, d+\frac12\})$,
\[
I_{72} \leq  (\e^\beta+\delta^{2})^{\frac{\alpha}{4}}\left(  \iint_{|x-y|<\delta}\phi^\e_\alpha(x-y)\rho^\e(x)\rho^\e(y)|u^\e(x)-u^\e(y)|^2\,dxdy\right)^{1/2}.
\]
Since $\delta < 1$, for $\alpha \in (0, \min \{ 2, d+\frac12\})$, and hence
\[
I_{72} \leq  (\e^\frac{\alpha\beta}4+\delta^{\frac\alpha2})\left(  \iint_{|x-y|<\delta}\phi^\e_\alpha(x-y)\rho^\e(x)\rho^\e(y)|u^\e(x)-u^\e(y)|^2\,dxdy\right)^{1/2}.
\] 	
We then choose $\delta = \e^{\frac{\beta}{2 + \alpha}} < 1$ to have
\begin{align*}
I_7 &\leq C\lt(\e^{\frac\beta2} + \e^\frac{\alpha\beta}4 + \e^\frac{\alpha\beta}{2(2+\alpha)}\rt)\lt(\int_0^t  \intoo\phi^\e_\alpha(x-y)\rho^\e(x)\rho^\e(y)|u^\e(x)-u^\e(y)|^2\,dxdyds \rt)^{\frac12}\cr
&\leq C\lt(\e^{\frac\beta2} + \e^\frac{\alpha\beta}4 + \e^\frac{\alpha\beta}{2(2+\alpha)}\rt)H(f^\e_0)^\frac12,
\end{align*}
due to \eqref{energy bound2}.
	
Combining all of the above estimates yields	
\[
\into \me(U^\e|U)\,dx \leq C\lt(\e^\frac12 + \e^{\frac\beta2} + \e^\frac{\alpha\beta}4 + \e^\frac{\alpha\beta}{2(2+\alpha)}\rt) + C\int_0^t \into \me(U^\e|U)\,dx ds
\]
for some $C>0$ independent of $\e > 0$, and applying Gr\"onwall's inequality concludes
\[
\into \me(U^\e|U)\,dx \leq C\lt(\e^\frac12 + \e^{\frac\beta2} + \e^\frac{\alpha\beta}4 + \e^\frac{\alpha\beta}{2(2+\alpha)}\rt).
\]

\begin{proof}[Proof of Theorem \ref{thm:hydrolimit}] The above estimate gives
\[
			\sup_{t \in [0,T] }\into \left(\rho^\e|u^\e-u|^2 + \mh(\rho^\e| \rho)\right) dx  \le C\e^\lambda,
\]
	where $\lambda:=\min\left\{\frac{1}{2}, \frac\beta2, \frac{\alpha \beta}{4}, \frac{\alpha\beta}{2(\alpha+2)}\right\}$. To show the strong convergence of $\rho^\e$ and $\rho^\e u^\e$, we observe from Lemma \ref{lem_lower_h} that 
\[
\rho^\e \to \rho \quad in \  L^\infty(0,T; L^r(\om)) \quad \mbox{as} \quad \e \to 0,
\]
where $r$ is given by
\[
r = \left\{ \begin{array}{ll}
 \frac{2}{3-\gamma}  & \textrm{if $d \geq 2$}\\ 
2 & \textrm{if $d=1$}
  \end{array} \right..
\]	
Note that $\frac{2}{3-\gamma} \leq \gamma$ for $d \geq 2$ and $\rho^\e$ is uniformly bounded in $L^1 \cap L^\gamma(\om)$. Thus the classical $L^p$ interpolation inequality yields
\[
\rho^\e \to\rho \quad \mbox{a.e.} \quad and \quad in \ L^\infty(0,T;L^p(\om))\quad \mbox{with $p \in (1,\gamma)$}.
\]	
For $r \in [1, \frac{d+2}{d+1})$, we get
\[
r < \gamma \quad \mbox{and} \quad \frac{r}{2-r} \in [1,\gamma).
\] 
This together with applying H\"older's inequality gives
\[
 \|\rho^\e(u^\e-u)\|_{L^r} \leq \lt(\into (\rho^\e)^{\frac r{2-r}}\,dx \rt)^{\frac{2-r}{2}}\lt(\into\rho^\e|u^\e-u|^2\,dx\rt)^{\frac r2} \to 0 \quad \mbox{as} \quad \e \to 0.
\]
We also find that for $r \in  (1, \frac{d+2}{d+1})$
\[
\|(\rho^\e-\rho)u\|_{L^r} \leq \|\rho^\e-\rho\|_{L^r}\|u\|_{L^\infty} \to 0 \quad \mbox{as} \quad \e \to 0
\]
and 
\[
\|(\rho^\e-\rho)u\|_{L^1} \leq \|\rho^\e - \rho\|_{L^p}\|u\|_{L^{p'}} \to 0 \quad \mbox{as} \quad \e \to 0,
\]
where $p$ is chosen such that 
\bq\label{pq}
p \in (1,\gamma) \quad \mbox{and} \quad p' = \frac p{p-1} \in ( \max\{ \frac\gamma{\gamma-1}, 2\}, \infty).
\eq
Combining the above two convergence estimates implies
 \begin{align*}
\|\rho^\e u^\e-\rho u\|_{L^r}&\le \|\rho^\e(u^\e-u)\|_{L^r}+\|(\rho^\e-\rho)u\|_{L^r} \to 0 \quad \mbox{as} \quad \e \to 0
\end{align*}
for $r \in [1, \frac{d+2}{d+1})$. 

We next estimate
\bq\label{conv3}
\lt|\intr v \otimes v f^\e\,dv - \rho u\otimes u - \rho^\gamma \mathbb{I}_d\rt|\leq \lt|\intr v \otimes v (f^\e - M[f^\e])\,dv  \rt| + \lt|\rho^\e u^\e \otimes u^\e - \rho u\otimes u\rt| + \lt| (\rho^\e)^\gamma - \rho^\gamma\rt|,
\eq
where the first term on the right hand side strongly converges in $L^1(\om \times (0,T))$ due to Lemma \ref{lem_gd}. On the other hand,
\[
 \lt|\rho^\e u^\e \otimes u^\e - \rho u\otimes u\rt| = \lt| \rho^\e(u^\e - u)\otimes u^\e + \rho^\e u \otimes (u^\e - u) + (\rho^\e - \rho) u\otimes u \rt|,
\]
and thus
\begin{align*}
&\into  \lt|\rho^\e u^\e \otimes u^\e - \rho u\otimes u\rt| dx\cr
&\quad \leq \lt(\into \rho^\e|u^\e - u|^2\,dx\rt)^\frac12 \lt( \lt(\into \rho^\e |u^\e|^2 \,dx\rt)^\frac12 + \lt(\into \rho^\e |u|^2 \,dx\rt)^\frac12 \rt)  + \into |\rho^\e - \rho| |u|^2\,dx\cr
&\quad \leq \lt(\into \rho^\e|u^\e - u|^2\,dx\rt)^\frac12 \lt( \lt(\intor |v|^2 f^\e_0  \,dxdv\rt)^\frac12 + \|u\|_{L^\infty} \rt) + \|u\|_{L^\infty} \|\rho^\e - \rho\|_{L^p}\|u\|_{L^{p'}}\cr
&\quad \to 0 \quad \mbox{as} \quad \e \to 0,
 \end{align*} 
 where we used \eqref{free_ineq}, $p$ and $p'$ are selected as in \eqref{pq}. For the convergence of the last term on the right hand side of \eqref{conv3}, we observe
\begin{align*} 
\into |(\rho^\e)^\gamma - \rho^\gamma|\,dx &= \int_{\rho^\e \geq \rho} \lt((\rho^\e)^\gamma - \rho^\gamma\rt)dx + \int_{\rho^\e \leq \rho} \lt(  \rho^\gamma - (\rho^\e)^\gamma\rt)dx \cr
&= (\gamma-1)\int_{\rho^\e \geq \rho} \mh(\rho^\e| \rho)\,dx + \gamma \into |\rho^\e - \rho| \rho^{\gamma-1}\,dx,
 \end{align*} 
 where we used \eqref{def_h}. It is clear that the first term on the right hand side of the above converges to $0$ as $\e \to 0$. For the second term, we use the similar argument as the above together with $\rho \in L^1\cap L^\infty(\om)$ to get
 \[
 \into |\rho^\e - \rho| \rho^{\gamma-1}\,dx \leq  \|\rho^\e - \rho\|_{L^p}\|\rho^{\gamma-1}\|_{L^{p'}} \to 0 \quad \mbox{as} \quad \e \to 0.
 \]
 Hence we have 
 \[
 \intr v \otimes v f^\e\,dv \to \rho u\otimes u + \rho^\gamma \mathbb{I}_d \quad \mbox{a.e.} \quad and \quad in \  L^1(\om \times (0,T))
 \]
 as $\e \to 0$. This completes the proof.
\end{proof}

%%%%%%%%%%%%%%%%%%%%%%%%%%%%%%%%%%%%%%%%%%%%%%%%%%%%%%%%%%%%%%
%
%
%.  
%.  
%
%
%%%%%%%%%%%%%%%%%%%%%%%%%%%%%%%%%%%%%%%%%%%%%%%%%%%%%%%%%%%%%%%%%%%%
\section{Global existence of weak solutions to the BGK-alignment model}\label{sec_weak}
In this section, we present the global existence of weak solutions to the BGK-alignment model \eqref{main_eq}. Note that the parameter $\e>0$ does not play any roles in estimating the existence theory, thus for notational simplicity, we set $\e = 1$, and we consider
\bq\label{weak_eq}
\pa_t f + v\cdot\nabla_x f  + \nabla_v \cdot \lt(F_{\phi}[f]f\rt)  = M[f] - f, 
\eq
where $M[f] =   \mathds{1}_{|u_f-v|^d\le c_d\rho_f}$ and 
\bq\label{asp_phi}
\phi \in W^{1,\infty}(\om) \mbox{ satisfying }|(\nabla \phi)(x)| \leq c_\phi \phi(x) \mbox{ for all }x \in \om
\eq 
for some $c_\phi>0$. We only consider the whole domain, i.e. $\om =\R^d$ in order to avoid the repetition. In fact, the case $\om = \T^d$ is easier than the whole space one.

%Then our notion of weak solutions and  main result are stated as follows.
%
%
%\newpage

\subsection{Regularized and linearized BGK-alignment model}
For the existence theory, we first regularize the equation \eqref{weak_eq} as
\bq\label{reg_weak_eq} 
\pa_t f^\K + v\cdot\nabla_x f^\K  + \nabla_v \cdot \lt(F_{\phi}[f^\K]f^\K\rt)  = M_\K[f^\K] * \varphi_\K - f^\K,
\eq
subject to regularized initial data:
\[
f^\K(x,v,t)|_{t=0} =: f^\K_0(x,v), \quad (x,v) \in \R^d \times \R^d,
\]
where
\[
F_{\phi}[f^\K] := \phi * (\rho_{f^\K} u_{f^\K}) - v \phi * \rho_{f^\K}
\]
and
\[
M_\K[f^\K]:=  \mathds{1}_{|u^{\K}_{f^\K}-v|^d\le c_d\rho^{\K}_{f^\K}}
\]
with
\[
\rho^{\K}_{f^\K} := \frac{\rho_{f^\K} * \theta_\K}{1 + \K^{d+1} \rho_{f^\K} * \theta_\K}, \quad  u^{\K}_{f^\K} := \frac{(\rho_{f^\K} u_{f^\K})*\theta_\K}{\rho_{f^\K} * \theta_\K + \K^{2d+1}(1 + |(\rho_{f^\K} u_{f^\K})*\theta_\K|^2)},  
\]
\[
\rho_{f^\K} = \intr f^\K\,dv,\quad \mbox{and}\quad \rho_{f^\K} u_{f^\K} = \intr vf^\K\,dv.
\] 
Here $\varphi_\K = \varphi_\K(x,v) = \theta_\K(x)\theta_\K(v)$ and $\theta_\K$ is  defined by means of the standard mollifier $\theta \in \mc^\infty_c$ as $\theta_\K(x) = \K^{-d}\theta(x/\K)$, so $\|\theta_\K\|_{L^1} = 1$.  The  regularized initial data $f^\K_0$ is defined by 
\[
f^\K_{0} =  f_0*\varphi_\K+\K \frac{e^{-|v|^2}}{1+|x|^q},\quad\mbox{with}\quad q> d.
\]
Note that in the case $\om = \T^d$ the second term on the right hand side of the above is unnecessary.

Throughout this section, the regularization parameter $\K$ is assumed to be less than  $1$. We construct the solution $f^\K$ to the regularized equation \eqref{reg_weak_eq} by considering the approximation sequence $f^{\K, n}$ given as solutions of the following equation:
\bq\label{app_weak_eq}
\pa_t f^{\K,n+1} + v\cdot\nabla_x f^{\K,n+1}  + \nabla_v \cdot \lt(F_{\phi}[f^{\K,n}]f^{\K,n+1}\rt)  = M_\K[f^{\K,n}]* \varphi_\K - f^{\K,n+1},
\eq
with the initial data and first iteration step:
\[
f^{\K,n}(x,v,t)|_{t=0} = f^\K_0(x,v) \quad \mbox{for all } n \geq 1 \quad \mbox{and} \quad f^{\K,0}(x,v,t) = f^\K_0(x,v) \quad \mbox{for } (x,v,t) \in \R^d \times \R^d \times (0,T).
\]
Here
\[
F_{\phi}[f^{\K,n}] := \phi * (\rho_{f^{\K,n}} u_{f^{\K,n}}) - v \phi * \rho_{f^{\K,n}}
\]
and
\[
M_\K[f^{\K,n}]:=  \mathds{1}_{|u^{\K}_{f^{\K,n}}-v|^d\le c_d\rho^{\K}_{f^{\K,n}}}
\]
with
\[
\rho^{\K}_{f^{\K,n}} = \frac{\rho_{f^{\K,n}} * \theta_\K}{1 + \K^{d+1} \rho_{f^{\K,n}} * \theta_\K}, \quad   u^{\K}_{f^{\K,n}} = \frac{(\rho_{f^{\K,n}} u_{f^{\K,n}})*\theta_\K}{\rho_{f^{\K,n}} * \theta_\K + \K^{2d+1}(1+ |(\rho_{f^{\K,n}} u_{f^{\K,n}})*\theta_\K|^2)}, 
\]
\[
\rho_{f^{\K,n}} = \intr f^{\K,n}\,dv, \quad \mbox{and} \quad \rho_{f^{\K,n}} u_{f^{\K,n}} = \intr vf^{\K,n}\,dv.
\]
In the following, for the sake of notational simplicity, we omit $\K$-dependence in $f^{\K,n}$, i.e., $f^n = f^{\K,n}$. In order to study the convergence of approximations $f^n$, motivated from \cite{CJpre}, we introduce a weighted $L^\infty$-norm:
\[
\|f\|_{L^\infty_{\ell}}:= \esssup_{(x,v) \in \R^d \times \R^d} e^{\lal v \ral^\ell} |f(x,v)|
\]
with $\ell > 0$, where $\lal v \ral := (1+ |v|^2)^{\frac12}$. Naturally, $L^\infty_\ell(\R^d \times \R^d)$ denotes the space of functions with finite norms.  For $s \in \N$, $W^{s,\infty}_\ell = W^{s,\infty}_\ell(\R^d \times \R^d)$ stands for $L^\infty_\ell$ Sobolev space of $s$-th order equipped with the norm:
\[
\|f\|_{W^{s,\infty}_{\ell}}:= \esssup_{(x,v) \in \R^d \times \R^d}  \sum_{|\alpha| + |\beta| \leq s} e^{\lal v \ral^\ell} |\pa^\alpha_x \pa^\beta_v f(x,v)|.
\]

For the regularized and linearized equation \eqref{app_weak_eq}, we show the global existence and uniqueness of solutions and uniform-in-$n$ bound estimates.
\begin{proposition}\label{prop_lin} Let $T>0$ and $\ell \in (1,2)$. Assume that $f_0$ satisfies the assumptions of Theorem \ref{thm_ext}. For any $n \in \N$, there exists a unique solution $f^n \in L^\infty(0,T; W^{1,\infty}_\ell(\R^d \times \R^d))$ of the equation \eqref{app_weak_eq} satisfying
\[
\sup_{n\in\N}\sup_{0 \leq t \leq T}\|f^n(\cdot,\cdot,t)\|_{L^1 \cap L^\infty} \leq C(\|f^\K_0\|_{L^1 \cap L^\infty} + 1),\qquad \inf_{n\in\N}\inf_{(x,t)\in \R^d\times(0,T)}\rho_{f^n}(x,t)\ge c_\K
\]
and
\[
\sup_{0 \leq t \leq T} \|f^n(\cdot,\cdot,t)\|_{W^{1,\infty}_\ell} \leq C_\K,
\]
where  $C>0$ is independent of both $\K$ and $n$, and $c_\K$ and $C_\K > 0$ depend on $\K$, but independent of $n$.
	 Moreover, we have 
\[
\sup_{n\in\N}\sup_{0 \leq t \leq T}\inttr |v|^2 f^n\,dxdv \leq   C_{T,f_0},
\]
which is indeed the uniform-in-$\K$ estimate on the kinetic energy.
\end{proposition}
\begin{proof} We first readily check the existence and uniqueness of solution $f^n \in L^\infty(0,T; W^{1,\infty}_\ell(\R^d \times \R^d))$ to the regularized and linearized equation \eqref{app_weak_eq} by the standard existence theory for transport equations, see \cite{HT08} for instance. Thus, we only provide the bound estimates on $\|f^n\|_{L^1\cap L^\infty}$ and $\|f^n\|_{W^{1,\infty}_\ell}$. 

\vspace{.2cm}

	\noindent $\bullet$ ($\|f^{n+1}\|_{L^1\cap L^\infty}$ estimate):  We begin with the estimate of $\|f^{n+1}\|_{L^1}$. Since
\[
\inttr M_\K[f^{n}]* \varphi_\K \,dxdv = \int_{\R^d} \rho^{\K}_{f^n} \,dx \leq \int_{\R^d} \rho_{f^n} * \theta_\K\,dx \leq \int_{\R^d} \rho_{f^n}\,dx,
\]
we obtain
\[
\frac{d}{dt}\inttr f^{n+1}\,dxdv \leq \inttr f^n\,dxdv - \inttr f^{n+1}\,dxdv.
\]
This  implies $\|f^n(t)\|_{L^1} \leq \|f^\K_0\|_{L^1}$ for all $n \geq 0$ and $t \in [0,T]$. 

For the estimate of $L^\infty$-norm of $f^n$, we introduce the following backward characteristics: 
\[
Z^{n+1}(s):= (X^{n+1}(s), V^{n+1}(s)):= (X^{n+1}(s;t,x,v), V^{n+1}(s;t,x,v)),
\]
which solves
\begin{align}\label{char}
\begin{aligned}
\frac{d}{ds} X^{n+1}(s) &= V^{n+1}(s),\cr
\frac{d}{ds} V^{n+1}(s) &= F_\phi[f^n](Z^{n+1}(s),s)
\end{aligned}
\end{align}
with the terminal data: 
\[
Z^{n+1}(t) = (x,v).
\]
Note that the above characteristics is well-defined due to $\phi \in W^{1,\infty}(\R^d)$. Along that characteristics, we have
\bq\label{f0}
\frac{d}{ds} f^{n+1}(Z^{n+1}(s),s) = \lt(d (\phi * \rho_{f^n})(X^{n+1}(s),s)-1 \rt)f^{n+1}(Z^{n+1}(s),s) + (M_\K[f^n] * \varphi_\K)(Z^{n+1}(s),s).
\eq
Thus,
\begin{align}\label{mild form}\begin{split}
 f^{n+1}(x,v,t) &= f^\K_0(Z^{n+1}(0)) \exp\lt(-\int_0^t (d (\phi * \rho_{f^n})(X^{n+1}(s),s)-1)\,ds \rt)\cr
 &\quad + \int_0^t (M_\K[f^n] * \varphi_\K)(Z^{n+1}(s),s)\exp\lt(-\int_s^t (d (\phi * \rho_{f^n})(X^{n+1}(\tau),\tau)-1)\,d\tau \rt)\,ds.
\end{split}\end{align}
Since $0 \leq  M_\K[f^n] \leq 1$ and $\|f^n(t)\|_{L^1} \leq \|f^\K_0\|_{L^1}$, we obtain $f^{n+1} \geq 0$ and 
\[
 f^{n+1}(x,v,t) \leq f^\K_0(Z^{n+1}(0)) e^{d\|\phi\|_{L^\infty}\|f^\K_0\|_{L^1}T} + Te^{d\|\phi\|_{L^\infty}\|f^\K_0\|_{L^1}T}. 
\]
Thus $\|f^{n+1}(t)\|_{L^\infty} \leq C(\|f^\K_0\|_{L^\infty} +1)$ for some $C>0$ independent of $n$ and $\K$. 

\vspace{.2cm}

\noindent $\bullet$ (Kinetic energy estimate): Note that \begin{align*}
&\frac12\frac{d}{dt} \inttr |v|^2 f^{n+1}\,dxdv + \frac12\inttr |v|^2 f^{n+1}\,dxdv \cr
&\quad = \inttr v \cdot F_{\phi}[f^n] f^{n+1}\,dxdv + \frac12\inttr |v|^2 (M_\K[f^n]* \varphi_\K)\,dxdv\cr
&\quad =: I_1 + I_2,
\end{align*}
where $I_1$ can be estimated as
\begin{align*}
I_1 &= \inttr v \cdot \phi * (\rho_{f^n} u_{f^n}) f^{n+1}\,dxdv-\inttr |v|^2 \phi * \rho_{f^n}  f^{n+1}\,dxdv \cr
&\leq \frac 12 \inttr   \phi * (\rho_{f^n} |u_{f^n}|^2) f^{n+1}\,dxdv-\frac 12\inttr |v|^2 \phi * \rho_{f^n}  f^{n+1}\,dxdv\cr
&\leq \frac 12 \|\phi * (\rho_{f^n} |u_{f^n}|^2)\|_{L^\infty}  \|f^{n+1}\|_{L^1}\cr
&\leq \frac 12\|\phi\|_{L^\infty}\|f^\K_0\|_{L^1} \intr  \rho_{f^n} |u_{f^n}|^2\,dx.  
\end{align*}
For $I_2$, we see that
\begin{align*}
I_2&\le\inttr |v|^2 M_\K[f^n]dxdv+\K^2\|f^n\|_{L^1} \cr
&\le \intr \rho^{\K}_{f^n} |u^{\K}_{f^n}|^2 + (\rho^{\K}_{f^n})^\gamma dx+\|f^\K_0\|_{L^1} \cr
&\leq \intr \rho_{f^n}  |u_{f^n} |^2 + (\rho_{f^n} )^\gamma dx+\|f^\K_0\|_{L^1},
\end{align*}
where we used
\[
\intr \rho^{\K}_{f^n} |u^{\K}_{f^n}|^2\,dx \leq \intr \frac{|(\rho_{f^n} u_{f^n}) * \theta_\K|^2}{\rho_{f^n} * \theta_\K}\,dx \leq \intr \rho_{f^n}  |u_{f^n} |^2 \,dx
\]
and
\[
\|\rho^{\K}_{f^n}\|_{L^\gamma} \leq \|\rho_{f^n} * \theta_\K\|_{L^\gamma} \leq \|\rho_{f^n}\|_{L^\gamma}.
\]
Combining the estimates for $I_1$ and $I_2$ with  \eqref{mini}, we obtain
\[
I_1+I_2 \leq C\left(\intr \rho_{f^n}  |u_{f^n} |^2 + (\rho_{f^n} )^\gamma\,dx +\|f^\K_0\|_{L^1}\right) \leq C\left(\inttr |v|^2 f^n\,dxdv+\|f^\K_0\|_{L^1}\right) 
\]
with $\gamma = 1+ \frac2d$. Therefore, we conclude that
$$
\frac{d}{dt} \inttr |v|^2 f^{n+1}\,dxdv + \inttr |v|^2 f^{n+1}\,dxdv \le   C\left(\inttr |v|^2 f^n\,dxdv+\|f^\K_0\|_{L^1}\right), 
$$
and subsequently, solving the above yields the uniform bound estimate on the kinetic energy of $f^{n+1}$.

\vspace{.2cm}

\noindent $\bullet$ (Lower bound on $\rho_{f^{n+1}}$): Since $M_\K *\phi_\K$ is positive, it can be easily obtained from \eqref{mild form} that
\begin{align}\begin{split}\label{lower bound of rho}
\rho_{f^{n+1}}(x,t) &\ge \intr f^\K_0(Z^{n+1}(0)) \exp\lt(-\int_0^t (d (\phi * \rho_{f^n})(X^{n+1}(s),s)-1)\,ds \rt)\,dv\cr
&\ge  e^{-(d\|\phi\|_{L^\infty}\|f^\K_0\|_{L^1}+1)T}\intr f^\K_0(Z^{n+1}(0)) \,dv\cr
&\ge  e^{-(d\|\phi\|_{L^\infty}\|f^\K_0\|_{L^1}+1)T}\K\intr \frac{e^{-|V^{n+1}(0)|^2}}{1+|X^{n+1}(0)|^q}  \,dv.
\end{split}\end{align}
On the other hand, solving \eqref{char} gives
\begin{align*}
V^{n+1}(s)&=\exp\lt(\int_s^t (\phi *\rho_{f^n})(X^{n+1}(\sigma),\sigma) \,d\sigma \rt)V^{n+1}(t)\cr
&-\int_s^t \exp\lt( \int_s^\tau(\phi *\rho_{f^n})(X^{n+1}(\sigma),\sigma) \,d\sigma\rt) (\phi *\rho_{f^n}u_{f^n}) (X^{n+1}(\tau),\tau) \, d\tau.
\end{align*}
Using the boundedness of $\phi$ and the uniform kinetic energy estimate of $f^n$, we find
\begin{align*}
|V^{n+1}(s)|&\le e^{\|\phi\|_{L^\infty}\|f_0^\K\|_{L^1}T}|v|+C_{T,f_0}e^{\|\phi\|_{L^\infty}\|f_0^\K\|_{L^1}T}\|\phi\|_{L^\infty},
\end{align*}
and hence
$$
e^{-|V^{n+1}(0)|^2} \ge C e^{-C|v|^2}.
$$
Also, we get
$$
|X^{n+1}(s)|\le |x|+\int_s^t|V^{n+1}(\tau)|\,d\tau \le C(1+|x|+|v|),
$$
yielding
$$
\frac{1}{1+|X^{n+1}(0)|^q} \ge \frac{C}{(1+|x|^q)(1+|v|^q)}.
$$
Therefore, going back to \eqref{lower bound of rho}, we conclude that
$$
\rho_{f^{n+1}}(x,t) \ge  \frac{C\K}{1+|x|^q}.
$$

\vspace{.2cm}

\noindent $\bullet$ ($\|f^{n+1}\|_{W^{1,\infty}_\ell}$ estimate): We use the characteristics defined in \eqref{char}. First, we have
\begin{align*}
\frac12\frac{d}{dt}\lt| e^{ \langle V^{n+1}(t)\rangle^\ell} f^{n+1}(Z^{n+1}(t),t)\rt|^2 & =\ell e^{2\langle V^{n+1}(t)\rangle^\ell} (f^{n+1}(Z^{n+1}(t),t))^2  \langle V^{n+1}(t) \rangle^{\ell -2} V^{n+1}(t)\cdot \frac{d V^{n+1}(t)}{dt}  \\
&\quad + e^{2\langle V^{n+1}(t)\rangle^\ell} (f^{n+1}(Z^{n+1}(t),t)) \frac{d}{dt}(f^{n+1}(Z^{n+1}(t),t))\\
&=: J_1 + J_2.
\end{align*}
Note that 
\[
v \cdot F_\phi[f^n] = v\cdot \phi * (\rho_{f^n} u_{f^n}) - |v|^2 \phi * \rho_{f^n} \leq \frac12\phi * (\rho_{f^n} |u_{f^n}|^2) - \frac{|v|^2}2 \phi * \rho_{f^n}. 
\]
Then, it follows from \eqref{mini} and the $L^1$-estimate of $(1+|v|^2)f^n$  that
\[
v \cdot F_\phi[f^n] \leq \|\phi\|_{L^\infty}\intr\rho_{f^n} |u_{f^n}|^2 \,dx -  \frac{|v|^2}2 \phi * \rho_{f^n}\leq C -  \frac{|v|^2}2 \phi * \rho_{f^n}.
\]
This gives
\begin{align*}
J_1 &= \ell e^{2\langle V^{n+1}(t)\rangle^\ell} (f^{n+1}(Z^{n+1}(t),t))^2  \langle V^{n+1}(t) \rangle^{\ell -2} V^{n+1}(t)\cdot F_\phi[f^n](Z^{n+1}(t),t)\cr
&\leq C_\K e^{2\langle V^{n+1}(t)\rangle^\ell} (f^{n+1}(Z^{n+1}(t),t))^2 \cr
&\quad - \frac \ell 2 e^{2\langle V^{n+1}(t)\rangle^\ell} (f^{n+1}(Z^{n+1}(t),t))^2 \phi * \rho_{f^n}(X^{n+1}(t),t) \langle V^{n+1}(t) \rangle^{\ell -2} |V^{n+1}(t)|^2. 
\end{align*}
For the estimate of $J_2$, we notice that 
\[
|e^{\langle v\rangle^\ell} M_\K[f^n](x,v)| \leq C_\K
\]
for some $C_\K >0$ independent of $n$, due to $\|\rho^{\K}_{f^n}\|_{L^\infty}, \|u^{\K}_{f^n}\|_{L^\infty} \leq C_\K$. This further implies
\[
|e^{\langle v\rangle^\ell} M_\K[f^n] * \varphi_\K| \leq C_\K \inttr e^{\langle v - w\rangle^\ell} M_\K[f^n](x-y,v-w) \varphi_\K(y,w)\,dydw \leq C_\K.
\]
This together with \eqref{f0} yields
\begin{align*}
J_2 &= e^{2\langle v \rangle^\ell} (f^{n+1}(z,t)) \lt( (d \phi * \rho_{f^n}(x) - 1)f^{n+1}(z,t) + M_\K[f^n] * \varphi(z) \rt)\bigg|_{(z=Z^{n+1}(t))}\cr
&\leq C e^{2\langle V^{n+1}(t)\rangle^\ell} (f^{n+1}(Z^{n+1}(t),t))^2 + C_\K e^{\langle V^{n+1}(t)\rangle^\ell} (f^{n+1}(Z^{n+1}(t),t))\cr
&\leq C_\K \lt(e^{2\langle V^{n+1}(t)\rangle^\ell} (f^{n+1}(Z^{n+1}(t),t))^2 + 1\rt).
\end{align*}
Hence we have
\[
\frac12\frac{d}{dt}\lt| e^{ \langle V^{n+1}(t)\rangle^\ell} f^{n+1}(Z^{n+1}(t),t)\rt|^2 \leq C_\K \lt(\lt| e^{ \langle V^{n+1}(t)\rangle^\ell} f^{n+1}(Z^{n+1}(t),t)\rt|^2 +1 \rt)
\]
and applying the Gr\"onwall's lemma to the above concludes
\bq\label{est_f0}
\sup_{0 \leq t \leq T}\|f^{n+1}(\cdot,t)\|_{L^\infty_\ell} \leq C_\K \|f^\K_0\|_{L^\infty_\ell} + C_\K,
\eq
where $C_\K >0$ is independent of $n$.

We next estimate the first-order derivative of $f^{n+1}$ in our weighted space. Note that $\pa_x f^{n+1}$ satisfies
\begin{align}\label{fx1}
\begin{aligned}
&\pa_t \pa_x f^{n+1} + v \cdot \nabla_x \pa_x f^{n+1} + F_\phi[f^n] \cdot \nabla_v \pa_x f^{n+1} \cr
&\quad = d(\pa_x \phi * \rho_{f^n}) f^{n+1} + (d\phi * \rho_{f^n} - 1)\pa_x f^{n+1} - \pa_x F_\phi [f^n] \cdot \nabla_v f^{n+1} + M_\K[f^n]*\pa_x \varphi_\K.
\end{aligned}
\end{align}
Then similarly as before, we estimate
\begin{align*}
&\frac12\frac{d}{dt}\lt| e^{ \langle V^{n+1}(t)\rangle^\ell} \pa_x f^{n+1}(Z^{n+1}(t),t)\rt|^2\\
&\quad =\ell e^{2\langle V^{n+1}(t)\rangle^\ell} |\pa_x f^{n+1}(Z^{n+1}(t),t)|^2  \langle V^{n+1}(t) \rangle^{\ell -2} V^{n+1}(t)\cdot \frac{d V^{n+1}(t)}{dt}  \\
&\qquad + e^{2\langle V^{n+1}(t)\rangle^\ell} (\pa_xf^{n+1}(Z^{n+1}(t),t)) \frac{d}{dt}(\pa_xf^{n+1}(Z^{n+1}(t),t))\\
&\quad =: K_1 + K_2,
\end{align*}
where
\begin{align*}
K_1 &\leq C_\K e^{2\langle V^{n+1}(t)\rangle^\ell}|\pa_x f^{n+1}(Z^{n+1}(t),t)|^2 \cr
&\quad - \frac \ell 2 e^{2\langle V^{n+1}(t)\rangle^\ell} |\pa_x f^{n+1}(Z^{n+1}(t),t)|^2 \phi * \rho_{f^n}(X^{n+1}(t),t) \langle V^{n+1}(t) \rangle^{\ell -2} |V^{n+1}(t)|^2. 
\end{align*}
For $K_2$, we observe that the first three terms on the right hand side of \eqref{fx1} can be bounded as
\begin{align*}
| d(\pa_x \phi * \rho_{f^n}) f^{n+1}| &\leq d\|\nabla \phi\|_{L^\infty} \|f^\K_0\|_{L^1}\|f^{n+1}\|_{L^\infty} \cr
&\leq C(\|f_0\|_{L^1} +1),\cr
|(d\phi * \rho_{f^n} - 1)\pa_x f^{n+1} | &\leq (d\|\phi\|_{L^\infty}\|f^\K_0\|_{L^1} + 1)|\pa_x f^{n+1} | \cr
&\leq C|\pa_x f^{n+1} |,
\end{align*}
and
\begin{align*}
|\pa_x F_\phi [f^n] \cdot \nabla_v f^{n+1}| &\leq \lt(\|\nabla \phi\|_{L^\infty}\|(1+|v|^2)f^n\|_{L^1} + c_\phi |v| \phi * \rho_{f^n} \rt)|\nabla_v f^{n+1}| \cr
&\leq  \lt(C + c_\phi |v| \phi * \rho_{f^n} \rt)|\nabla_v f^{n+1}|,
\end{align*}
where we used the assumption on $\phi$ in \eqref{asp_phi}. Then we obtain
\begin{align*}
K_2 &\leq C_\K \lt(e^{2\langle V^{n+1}(t)\rangle^\ell} |\pa_xf^{n+1}(Z^{n+1}(t),t)|^2 + 1\rt) \cr
&\quad + c_\phi  e^{2\langle V^{n+1}(t) \rangle^\ell} \phi * \rho_{f^n}(X^{n+1}(t),t)  |\pa_xf^{n+1}(Z^{n+1}(t),t)||V^{n+1}(t)|  |\nabla_v f^{n+1}(Z^{n+1}(t),t)|.
\end{align*}
In order to handle the second term, we use
\[
2^{\frac{\ell -2}{2}}|v|^\ell \le \langle v\rangle^{\ell -2} |v|^2, \quad \mbox{for}\quad |v|\ge 1\quad\mbox{and} \quad \ell  \in (1,2),
\]
to get
\begin{align*}
\phi * \rho_{f^n}   |\pa_xf^{n+1}| |v|  |\nabla_v f^{n+1}| & \leq (\phi * \rho_{f^n} ) |\pa_xf^{n+1}|  |\nabla_v f^{n+1}| \mathds{1}_{\{|v| \leq 1\}} \cr
&\quad + \delta \lt( (\phi * \rho_{f^n} )^{\frac1\ell }   |\pa_xf^{n+1}| |v|  |\nabla_v f^{n+1}|^{\frac2\ell  - 1}\rt)^\ell \mathds{1}_{\{|v| \geq 1\}} \cr
&\quad + \delta \lt( (\phi * \rho_{f^n} )^{1-\frac1\ell }     |\nabla_v f^{n+1}|^{2-\frac2\ell }\rt)^\frac{\ell }{\ell -1} \mathds{1}_{\{|v| \geq 1\}} \cr
&\leq C_\K \lt(|\pa_xf^{n+1}|^2 +  |\nabla_v f^{n+1}|^2 \rt)\cr
&\quad + \delta (\phi * \rho_{f^n} )\lal v \ral^{\ell -2}|v|^2 \lt(\frac{\ell }{2^\frac{\ell }{2}} |\pa_xf^{n+1}|^2  + \frac{2-\ell }{2^\frac{\ell }{2}}  |\nabla_v f^{n+1}|^2\rt),
\end{align*}
where $\delta > 0$ will be determined later. This together with \eqref{est_f0} yields
\begin{align*}
K_2 &\leq C_\K e^{2\langle V^{n+1}(t)\rangle^\ell}  \lt(   |\pa_x f^{n+1}(Z^{n+1}(t),t)|^2  + |\nabla_v f^{n+1}(Z^{n+1}(t),t)|^2\rt) + C_\K \cr
&\quad + c_\phi \delta e^{2\langle V^{n+1}(t) \rangle^\ell} \phi * \rho_{f^n}(X^{n+1}(t),t) \lal V^{n+1}(t) \ral^{\ell -2}|V^{n+1}(t)|^2\cr
&\hspace{3cm} \times \lt(\frac{\ell }{2^\frac{\ell }{2}} |\pa_xf^{n+1}(Z^{n+1}(t),t)|^2  + \frac{2-\ell }{2^\frac{\ell }{2}}  |\nabla_v f^{n+1}(Z^{n+1}(t),t)|^2\rt).
\end{align*}
Thus, we obtain
\begin{align}\label{est_fx1}
\begin{aligned}
&\frac12\frac{d}{dt}\lt| e^{ \langle V^{n+1}(t)\rangle^\ell} \pa_x f^{n+1}(Z^{n+1}(t),t)\rt|^2\\
&\quad + \frac \ell 2\lt(1 -  \delta c_\phi 2^{1 - \frac \ell 2} \rt) e^{2\langle V^{n+1}(t)\rangle^\ell} |\pa_x f^{n+1}(Z^{n+1}(t),t)|^2 \phi * \rho_{f^n}(X^{n+1}(t),t) \langle V^{n+1}(t) \rangle^{\ell -2} |V^{n+1}(t)|^2\cr
&\qquad \leq C_\K e^{2\langle V^{n+1}(t)\rangle^\ell} \lt(  |\pa_x f^{n+1}(Z^{n+1}(t),t)|^2  + |\nabla_v f^{n+1}(Z^{n+1}(t),t)|^2\rt) + C_\K  \cr
&\quad \qquad + \delta  c_\phi  \frac{2-\ell }{2^\frac{\ell }{2}} e^{2\langle V^{n+1}(t) \rangle^\ell} \phi * \rho_{f^n}(X^{n+1}(t),t) \lal V^{n+1}(t) \ral^{\ell -2}|V^{n+1}(t)|^2|\nabla_v f^{n+1}(Z^{n+1}(t),t)|^2 .
\end{aligned}
\end{align}
For the estimate of $\|\pa_v f^{n+1}\|_{L^\infty_\ell}$, we notice that $\pa_v f^{n+1}$ satisfies
\begin{align*}
\pa_t \pa_v f^{n+1} + v \cdot \nabla_v \pa_x f^{n+1} + F_\phi[f^n] \cdot \nabla_v \pa_v f^{n+1}  = -\pa_x f^{n+1} + ((d + 1)\phi * \rho_{f^n} - 1)\pa_v f^{n+1}  + M_\K[f^n]*\pa_v \varphi_\K
\end{align*}
and by using a similar argument as above, we deduce 
\begin{align}\label{est_fv1}
\begin{aligned}
&\frac12\frac{d}{dt}\lt| e^{ \langle V^{n+1}(t)\rangle^\ell} \pa_v f^{n+1}(Z^{n+1}(t),t)\rt|^2\\
&\quad + \frac \ell 2  e^{2\langle V^{n+1}(t)\rangle^\ell} |\pa_v f^{n+1}(Z^{n+1}(t),t)|^2 \phi * \rho_{f^n}(X^{n+1}(t),t) \langle V^{n+1}(t) \rangle^{\ell -2} |V^{n+1}(t)|^2\cr
&\qquad \leq C_\K  e^{2\langle V^{n+1}(t)\rangle^\ell} \lt( |\pa_x f^{n+1}(Z^{n+1}(t),t)|^2  + |\pa_v f^{n+1}(Z^{n+1}(t),t)|^2\rt) + C_\K. 
\end{aligned}
\end{align}
We now combine \eqref{est_fx1} and \eqref{est_fv1} and choose $\delta > 0$ small enough to have 
\[
\frac{d}{dt}\| f^{n+1}(\cdot,\cdot,t)\|_{W^{1,\infty}_\ell}^2 \leq C_\K \| f^{n+1}(\cdot,\cdot,t)\|_{W^{1,\infty}_\ell}^2 + C_\K
\]
and hence
\[
\sup_{0 \leq t \leq T}\| f^{n+1}(\cdot,\cdot,t)\|_{W^{1,\infty}_\ell} \leq C_\K   \|f^\K_0\|_{W^{1,\infty}_\ell} + C_\K.
\]
\end{proof}

\subsection{Cauchy estimates} Next, we show that $\lal v \ral^2 f^n$, where $f^n$ is the solution of \eqref{app_weak_eq}, is Cauchy in  $L^\infty(0,T; L^1(\R^d \times \R^d))$. For this, we observe that

\begin{align*}
& \frac{d}{dt}\inttr \lal v \ral^2 |f^{n+1} - f^n|\,dxdv+\inttr \lal v \ral^2 |f^{n+1} - f^n|\,dxdv \cr
&\quad = - \inttr \lal v \ral^2 {\rm sgn}(f^{n+1} - f^n) \nabla_v \cdot (F_\phi[f^n] f^{n+1} - F_\phi[f^{n-1}]f^n)\,dxdv\cr
&\qquad +  \inttr \lal v \ral^2 {\rm sgn}(f^{n+1} - f^n) (M_\K[f^n] - M_\K[f^{n-1}])*\varphi_\K\,dxdv\cr
&\quad =: I_1 + I_2.
\end{align*}
For the estimate of $I_1$, we divide it into three terms:
\begin{align*}
I_1 &= - \inttr \lal v \ral^2 {\rm sgn}(f^{n+1} - f^n) \nabla_v \cdot ((F_\phi[f^n] - F_\phi[f^{n-1}])f^{n+1})\,dxdv\cr
&\quad + 2\inttr (v \cdot F_\phi[f^{n-1}]) |f^{n+1} - f^n|\,dxdv\cr
&=d \inttr \lal v \ral^2 {\rm sgn}(f^{n+1} - f^n) (\phi * (\rho_{f^n} - \rho_{f^{n-1}}))f^{n+1}\,dxdv\cr
&\quad - \inttr \lal v \ral^2 {\rm sgn}(f^{n+1} - f^n)  (F_\phi[f^n] - F_\phi[f^{n-1}])\cdot \nabla_v f^{n+1}\,dxdv \cr
&\quad + 2\inttr (v \cdot \phi * (\rho_{f^n} u_{f^n}) - |v|^2 \phi * \rho_{f^n}) |f^{n+1} - f^n|\,dxdv\cr
&=:I_1^1+I_1^2+ I_1^3,
\end{align*}
where
\[
I_1^1 \leq d\|\phi\|_{L^\infty}\|\rho_{f^n} - \rho_{f^{n-1}}\|_{L^1} \inttr \lal v \ral^2 f^{n+1}\,dxdv \leq C_\K \inttr  |f^n - f^{n-1}|\,dxdv,
\]
\begin{align*}
I_1^2  &\leq \|\phi * (\rho_{f^n} u_{f^n} - \rho_{f^{n-1}} u_{f^{n-1}})\|_{L^\infty} \inttr \lal v \ral^2 |\nabla_v f^{n+1}|\,dxdv\cr
&\quad + \|\phi * (\rho_{f^n}   - \rho_{f^{n-1}}  )\|_{L^\infty} \inttr \lal v \ral^3 |\nabla_v f^{n+1}|\,dxdv\cr
&\leq C  \lt(\|\rho_{f^n} u_{f^n} - \rho_{f^{n-1}} u_{f^{n-1}}\|_{L^1} + \|\rho_{f^n}   - \rho_{f^{n-1}} \|_{L^1} \rt)\cr
&\leq C   \inttr \lal v \ral^2 |f^n - f^{n-1}|\,dxdv
\end{align*}
due to Proposition \ref{prop_lin}, and
\begin{align*}
I_1^3  &\leq 2 (\|\phi * (\rho_{f^n} u_{f^n})\|_{L^\infty}+\|\phi * \rho_{f^n} \|_{L^\infty}) \inttr \lal v\ral^2 |f^{n+1} - f^n|\,dxdv\cr
&\leq C \inttr \lal v \ral^2 |f^{n+1} - f^n|\,dxdv.
\end{align*}
We then use \cite[Lemma 2.1]{CHpre} to estimate
\begin{align*}
I_2 &\leq \inttr \lal v \ral^2  |(M_\K[f^n] - M_\K[f^{n-1}])*\varphi_\K|\,dxdv\cr
&\leq C_\K \inttr \lal v \ral^2  |M_\K[f^n] - M_\K[f^{n-1}]|\,dxdv\cr
&\leq C_\K \lt(\|\rho^{\K}_{f^n} - \rho^{\K}_{f^{n-1}}\|_{L^1} +\|u^{\K}_{f^n} - u^{\K}_{f^{n-1}}\|_{L^1}\rt).
\end{align*}
Since the uniform-in-$n$ boundedness of $\|f^n\|_{W^{1,\infty}_\ell}$ in Proposition \ref{prop_lin} gives the boundedness of $\rho_{f^n} u_{f^n}$ uniformly in $n$, we can further estimate
\[
|\rho^{\K}_{f^n} - \rho^{\K}_{f^{n-1}}| \leq C_\K |(\rho_{f^n}   - \rho_{f^{n-1}}  )* \theta_\K|
\]
and
\[
|u^{\K}_{f^n} - u^{\K}_{f^{n-1}}| \leq C_\K\lt( |(\rho_{f^n} u_{f^n} - \rho_{f^{n-1}} u_{f^{n-1}})* \theta_\K| + |(\rho_{f^n}   - \rho_{f^{n-1}}  )* \theta_\K|  \rt).
\]
Thus, we obtain
\[
I_2 \leq C_\K \inttr \lal v \ral^2 |f^n - f^{n-1}|\,dxdv,
\]
where $C_\K > 0$ is independent of $n$. Hence $\lal v \ral^2 f^n$ is a Cauchy sequence in $L^\infty(0,T; L^1(\R^d \times \R^d))$. Thus, for a fixed $\K>0$ there exists a limiting function $\lal v \ral^2 f \in L^\infty(0,T; L^1(\R^d \times \R^d))$ such that 
\begin{equation}\label{conv_0}
\sup_{0 \leq t \leq T} \|\lal v \ral^2(f^n - f)(\cdot,\cdot,t)\|_{L^1} \to 0 \quad \mbox{as} \quad n \to \infty.
\end{equation}
From this, one can deduce that
\begin{equation}\label{conv_1}
\sup_{0\le t\le T}(\|(\rho_{f^n}-\rho_f)(\cdot,t)\|_{L^p}+\|(\rho_{f^n} u_{f^n}-\rho_{f}u_f)(\cdot,t)\|_{L^p})\rightarrow 0 \quad \mbox{as}\quad n\rightarrow \infty
\end{equation}
for any $p\in [1,\infty)$.

\subsection{Existence of weak solutions to the regularized equation}  
 In this subsection, we prove that the limiting function $f$ is in fact the weak solution to the regularized equation \eqref{reg_weak_eq} in the sense of Definition \ref{def_weak}. We then provide the kinetic energy estimate:
$$\begin{aligned}
&\frac12\inttr |v|^2 f\,dxdv + \frac12\int_0^t \inttr |v|^2 (f - M_\K[f]*\varphi_\K)\,dxdvds  \cr
&\quad = \int_0^t  \iint_{\R^{d} \times \R^{d}} f v\cdot F_{\phi}[f]  \,dxdvds + \frac12\inttr |v|^2 f^\K_0\,dxdv.
\end{aligned}$$
We start with the weak formulation of \eqref{reg_weak_eq} that for any  $\eta \in \mc^1_c(\R^d \times \R^d \times [0,T])$ with $\eta(x,v,T) = 0$,
	\begin{align*}
&- \iint_{\R^d\times\R^d} f_{0}^\K \eta(x,v,0)\,dxdv - \int_0^T \iint_{\R^d\times\R^d} f^{n+1} (\pa_t \eta + v \cdot \nabla_x \eta+F_{\phi}[f^n]\cdot\nabla_v\eta)\,dxdvdt \cr
&\qquad= \int_0^T \iint_{\R^d\times\R^d} \lt(M_\K[f^n]*\varphi_\K - f^{n+1}\rt)\eta\,dxdvdt.
\end{align*}

Since the first three terms on the left hand side and the second term on the right hand side are linear, it suffices to deal with terms with the velocity alignment and the equilibrium function. We first observe that
\begin{align*}
&\iint_{\R^d\times\R^d} |M_\K[f^n]*\varphi_\K\eta-M_\K[f]*\varphi_\K \eta| \,dxdv\cr
&\quad = \iiiint_{\R^{2d}\times\R^{2d}} \{M_\K[f^n]-M_\K[f]\}(x-y,v-w)\varphi_\K(y,w)\eta(x,v) \,dxdydvdw \cr
&\le  \|\eta\|_{L^\infty}\iint_{\R^d\times\R^d}|M_\K[f^n]-M_\K[f]|\,dxdv\iint_{\R^d\times\R^d}   \varphi_\K(y,w) \,dydw\cr
&\le C\lt(\|\rho^{\K}_{f^n}-\rho^\K_f\|_{L^1}+\|u^{\K}_{f^n}-u^\K_f\|_{L^1}\rt)
\end{align*}
where we used \cite[Lemma 2.1]{CHpre}. This together with \eqref{conv_1} gives the desired result. Now it only remains to show that
\begin{align*}
\int_0^T \iint_{\R^d\times\R^d} f^{n+1} F_{\phi}[f^n]\cdot\nabla_v\eta\,dxdvdt\rightarrow\int_0^T \iint_{\R^d\times\R^d} f F_{\phi}[f]\cdot\nabla_v\eta\,dxdvdt\quad\mbox{as}\quad n\rightarrow \infty.
\end{align*}
Observe that
\begin{align}\label{alignment}\begin{split}
& \iint_{\R^d\times\R^d} (f^{n+1} F_{\phi}[f^n]-f F_{\phi}[f])\cdot\nabla_v\eta\,dxdv\cr
&\quad =  \iint_{\R^d\times\R^d} (f^{n+1} (\phi * (\rho_{f^n} u_{f^n}) - v \phi * \rho_{f^n})-f (\phi * (\rho u) - v \phi * \rho))\cdot\nabla_v\eta\,dxdv\cr
&\quad =  \iint_{\R^d\times\R^d} f^{n+1} \phi * (\rho_{f^n} u_{f^n}-\rho u)  \cdot\nabla_v\eta\,dxdv+\iint_{\R^d\times\R^d}(f^{n+1} -f  )\phi * (\rho u)\cdot\nabla_v\eta\,dxdv\cr
&\qquad - \iint_{\R^d\times\R^d} vf^{n+1}  \phi * (\rho_{f^n}-    \rho)\cdot\nabla_v\eta\,dxdv- \iint_{\R^d\times\R^d} v(f^{n+1}   -f   )\phi * \rho\cdot\nabla_v\eta\,dxdv.
\end{split}\end{align}
This yields
\begin{align*}
 &\iint_{\R^d\times\R^d} (f^{n+1} F_{\phi}[f^n]-f F_{\phi}[f])\cdot\nabla_v\eta\,dxdv\cr
 &\quad \le   C_\K(\|\rho_{f^n} u_{f^n}-\rho_f u_f \|_{L^1}+\|\rho_{f^n}-    \rho_f\|_{L^1}+  \|\lal v\ral ^2 (f^{n+1}   -f   )\|_{L^1}).
\end{align*}
We now use \eqref{conv_0} and \eqref{conv_1} to obtain the desired result. 

Next, we will prove the kinetic energy estimate. We see that
\begin{align*}
&\frac 12\iint_{\R^d\times\R^d}|v|^2f^{n+1}(x,v,t)\,dxdv+\frac 12\int_0^t\iint_{\R^d\times\R^d}|v|^2 f^{n+1}\,dxdvds\cr
&\quad = \int_0^t  \iint_{\T^{d} \times \R^{d}}  v\cdot F_{\phi}[f^n]f^{n+1}  \,dxdvds + \frac 12\int_0^t\iint_{\R^d\times\R^d}|v|^2 (M_\K [f^n]* \varphi_\K)\,dxdvds \cr
&\qquad +\frac12\inttr |v|^2 f_0\,dxdv,
\end{align*}
thus it suffices to deal with the first two terms on the right hand side. In the same manner as in \eqref{alignment}, we get
\begin{align*}
& \iint_{\R^d\times\R^d} v\cdot (f^{n+1} F_{\phi}[f^n]-f F_{\phi}[f])\,dxdv\cr
&\quad =  \iint_{\R^d\times\R^d} f^{n+1} \phi * (\rho_{f^n} u_{f^n}-\rho u)  \cdot v\,dxdv+\iint_{\R^d\times\R^d}(f^{n+1} -f  )\phi * (\rho u)\cdot v\,dxdv\cr
&\qquad - \iint_{\R^d\times\R^d} |v|^2f^{n+1}  \phi * (\rho_{f^n}-    \rho)\,dxdv- \iint_{\R^d\times\R^d} |v|^2(f^{n+1}   -f   )\phi * \rho\,dxdv,
\end{align*}
which, combined with Proposition \ref{prop_lin}, leads to
$$
 \iint_{\R^d\times\R^d} v\cdot f^{n+1} F_{\phi}[f^n]\,dxdv\rightarrow  \iint_{\R^d\times\R^d} v\cdot f F_{\phi}[f]\,dxdv\quad\mbox{as}\quad n\rightarrow \infty.
$$
For the estimate of second term, we see that
\begin{align*}
&\iint_{\R^d\times\R^d}|v|^2 ((M_\K [f^n]-M_\K [f])* \varphi_\K)\,dxdv\cr
&\quad =\iint_{\R^{2d}\times\R^{2d}}|v|^2 (M_\K [f^n]-M_\K [f])(x-y,v-w) \varphi_\K(y,w)\,dydxdwdv\cr
&\quad \le \iint_{\R^{2d}\times\R^{2d}}|v-w|^2 |M_\K [f^n]-M_\K [f]|(x-y,v-w) \varphi_\K(y,w)\,dydxdwdv\cr
&\qquad + \iint_{\R^{2d}\times\R^{2d}}|w|^2 |M_\K [f^n]-M_\K [f]|(x-y,v-w) \varphi_\K(y,w)\,dydxdwdv\cr
&\quad \le C_\K(\|\rho^{\K}_{f^n}-\rho^{\K}_f \|_{L^1}+\|u^{\K}_{f^n}-u^\K_f \|_{L^1}).
\end{align*}
In the last line, we used  \cite{CHpre}. Thus we conclude that
$$
\iint_{\R^d\times\R^d}|v|^2 M_\K [f^n]* \varphi_\K\,dxdv\rightarrow \iint_{\R^d\times\R^d}|v|^2 M_\K [f]* \varphi_\K\,dxdv\quad \mbox{as}\quad n\rightarrow \infty
$$
which completes the proof.

\begin{remark}

Here we reveal the uniform-in-$\K$ bound estimate on kinetic energy of the limiting function $f^\K$ for clarity. Even though it can be directly obtained from Proposition \ref{prop_lin}, we revisit it once again to more clearly see the relevance to the initial data. For this, we note that
\begin{align*}
&\frac12\frac{d}{dt} \inttr |v|^2 f^\K\,dxdv + \frac12\inttr |v|^2 f^\K\,dxdv \cr
&\quad = \inttr v \cdot F_{\phi}[f^\K] f^\K\,dxdv + \frac12\inttr |v|^2 (M_\K[f^\K]* \varphi_\K)\,dxdv\cr
&\quad =: I_1 + I_2,
\end{align*}
where $I_1$ can be estimated as
\[
I_1 = -\frac12\iiiint_{\R^{2d} \times \R^{2d}} \phi(x-y)|v-w|^2 f^\K(x,v)f^\K(y,w)\,dxdydvdw \leq 0.
\]
For $I_2$, we notice from \eqref{mini} that
\[
\inttr |v|^2 M_\K[f^\K]\,dxdv = \intr \rho^\K_{f^\K} |u^{\K}_{f^\K}|^2 + (\rho^{\K}_{f^\K})^\gamma\,dx \leq \intr \rho_{f^\K}  |u_{f^\K} |^2 + (\rho_{f^\K})^\gamma\,dx \leq \inttr |v|^2 f^\K\,dxdv
\]
with $\gamma = 1+ \frac2d$. This yields
\[
I_2 \leq \inttr (|v|^2 M_\K[f^\K])* \varphi_\K\,dxdv + \K^2 \leq \inttr |v|^2 f^\K\,dxdv + \K^2.
\]
We then now combine all of the above estimates and apply the Gr\"onwall's lemma to have
\[
\inttr |v|^2 f^\K\,dxdv \leq   C\inttr |v|^2 f^\K_0\,dxdv + C\K^2
\]
for all $0 \leq t \leq T$, where $C>0$ is independent of $\K>0$.
\end{remark}

\subsection{Proof of Theorem \ref{thm_ext}} We now pass to the limit $\K \to 0$ and show that $f:=\lim_{\K \to 0} f^\K$ satisfies the equation \eqref{weak_eq} in the sense of Definition \ref{def_weak} and the kinetic energy inequality \eqref{free_ineq}.

We first present a lemma, showing some relationship between the local density and the kinetic energy, which will be used to estimate the interaction energy. Since this lemma is by now classical \cite{GS99}, we skip its proof. 

\begin{lemma}\label{lem_tech} Suppose $f \in L^1_+ \cap L^\infty(\Omega \times \R^d)$  and $|v|^2 f \in L^1(\Omega \times \R^d)$. Then there exists a constant $C>0$ such that
	\[
	\|\rho_f\|_{L^\frac{d+2}d} \leq C\left\|f\right\|_{L^\infty}^{ \frac{d+2}2} \lt(\iint_{\R^d \times \R^d} |v|^2 f\,dxdv\rt)^{\frac d{d+2}}
	\]
	and
	\[
	\|\rho_f u_f\|_{L^\frac{d+2}{d+1}} \leq C\left\|f\right\|_{L^\infty}^{d+2} \lt(\iint_{\R^d \times \R^d} |v|^2 f\,dxdv\rt)^{\frac {d+1}{d+2}}.
	\]
\end{lemma}

Using the uniform bound estimates and Lemma \ref{lem_tech}, we obtain that there exists $f \in L^\infty (\R^d \times \R^d \times (0,T))$ such that
\[%\begin{equation}\label{f weak star}
f^\K \stackrel{\ast}{\rightharpoonup} f \quad \mbox{in } L^\infty (\R^d \times \R^d \times (0,T)), \quad \rho_{f^\K} \stackrel{\ast}{\rightharpoonup} \rho_f \quad \mbox{in } L^\infty (0,T; L^p(\R^d)) \quad \mbox{with } p \in \lt[1,\frac{d+2}{d}\rt],
\]%\end{equation}
and
\[%\begin{equation}\label{weak star 2}
\rho_{f^\K} u_{f^\K} \stackrel{\ast}{\rightharpoonup} \rho_f u_f \quad \mbox{in } L^\infty (0,T; L^q(\R^d)) \quad \mbox{with } q \in \lt[1,\frac{d+2}{d+1}\rt].
\]%\end{equation}
On the other hand, we know $M_\K[f^\K] \in L^\infty(\R^d \times \R^d \times (0,T))$, and thus $M_\K[f^\K] \in L^p_{loc}(\R^d \times \R^d \times (0,T))$. Moreover, 
\begin{align}\label{sp_mom}
\begin{aligned}
&\frac{d}{dt}\iint_{\R^d\times \R^d} |x|^2 f^\K\,dxdv + \iint_{\R^d\times \R^d} |x|^2 f^\K\,dxdv \cr
&\quad = 2\iint_{\R^d\times \R^d} x \cdot v f^\K\,dxdv + \iint_{\R^d\times \R^d} |x|^2 M_\K[f^\K]*\varphi_\K\,dxdv\cr
&\quad \leq 3\iint_{\R^d\times \R^d} |x|^2 f^\K\,dxdv + \iint_{\R^d\times \R^d} |v|^2 f^\K\,dxdv + 2\K,
\end{aligned}
\end{align}
where we used
\begin{align*}
\iint_{\R^d\times \R^d} |x|^2 M_\K[f^\K]*\varphi_\K\,dxdv &\le  (1+\K^2)\iint_{\R^d\times \R^d} (1+|x|^2) M_\K[f^\K]\,dxdv \cr
&\le 2 \intr (1+|x|^2) \rho^\K_{f^\K} \,dx\cr
&\leq \intr (1+|x|^2)  (\rho_{f^\K}  * \theta_\K)\,dx \leq 2\intr (1+|x|^2) \rho_{f^\K} \,dx + 2\K.
\end{align*}
Since the kinetic energy is uniformly bounded in $\K>0$, applying the Gr\"onwall's lemma to \eqref{sp_mom} gives
\[
\iint_{\R^d\times \R^d} |x|^2 f^\K\,dxdv < \infty
\]
uniformly in $\K>0$. Then we now apply the strong compactness lemma \cite[Lemma 2.6]{KMT13}, based on the velocity averaging, to get
\bq\label{strong_comp}
\rho_{f^\K}  \to \rho_f \quad  \mbox{in } L^\infty (0,T; L^p(\R^d)) \quad \mbox{and} \quad   \rho_{f^\K}  u_{f^\K} \to \rho_f u_f \quad  \mbox{in } L^\infty (0,T; L^p(\R^d))
\eq
for $p \in [1,\frac{d+2}{d+1})$. On the other hand, it follows from \eqref{strong_comp} that
\bq\label{conv_ae0}
\|\rho_{f^\K}  * \theta_\K - \rho_f\|_{L^1} \leq \|(\rho_{f^\K} - \rho_f)*\theta_\K\|_{L^1} + \|\rho_{f} * \theta_\K - \rho_f\|_{L^1} \to 0 \quad \mbox{as} \quad \K \to 0,
\eq
Similarly, $\|(\rho_{f^\K} u_{f^\K}) * \theta_\K - \rho_f u_f\|_{L^1}\rightarrow 0$ as $\K \to 0$. Thus combining that with 
\[
\K^{d+1}|\rho_{f^\K} * \theta_\K| \leq C\K \quad \mbox{and} \quad \K^{2d+1}|(\rho_{f^\K}u_{f^\K}) * \theta_\K|^2 \leq C\K
\]
deduces that 
\bq\label{conv_ae}
\rho^{\K}_{f^\K} \to \rho_f  \quad \mbox{and} \quad u^{\K}_{f^\K} \to u_f  \quad \mbox{a.e. on } E 
\eq
where $E := \{(x,t)\in \R^d\times [0,T] : \rho_f(x,t) > 0\}$. 
Indeed,
\begin{align*}
|\rho^{\K}_{f^\K} - \rho_f |&=\left|\frac{\rho_{f^\K}*\theta_\K -(1+\K^{d+1}\rho_{f^\K}*\theta_\K)\rho_f}{1+\K^{d+1}\rho_{f^\K}*\theta_\K}\right|\cr
&=\left|\frac{\rho_{f^\K}*\theta_\K-\rho_f -\K^{d+1}\rho_f(\rho_{f^\K}*\theta_\K)}{1+\K^{d+1}(\rho_{f^\K}*\theta_\K)}\right|\to 0\quad \mbox{a.e.}\quad \mbox{on}\quad E
\end{align*}
due to \eqref{conv_ae0}. We now show that the limiting function $f$ satisfies our main equation in the distributional sense. For this, it is sufficient to deal only with terms related to the equilibrium function and velocity alignment since the other terms are linear.  We first observe that 
	\begin{align*}
	M_\K[f^\K]*\varphi_\K-M[f] &=(M_\K[f^\K]-M[f])*\varphi_\K+(M[f]*\varphi_\K-M[f]) \cr
	&=:I_1+I_2.
	\end{align*}
Note  that $M[f]*\varphi_\K$ converges to $M[f]$ a.e. as $\K\rightarrow 0$, and for any $p\in(1,\infty)$,
$$
\|M[f]*\varphi_\K\|_{L^p}\le \|M[f]\|_{L^p}\|\varphi_\K\|_{L^1}=\|M[f]\|_{L^1}^{\frac 1p}=\|f\|_{L^1}^{\frac 1p} <C.
$$
Thus one can see that $M[f]*\varphi_\K$ weakly converges to $M[f]$ in $L^p$, and this leads to
$$
\iint_{E\times \R^d} I_2\eta\,dxdvdt\rightarrow 0\quad\mbox{as}\quad\K\to0
$$
for any test function $\eta \in \mc^1_c(\R^d \times \R^d \times [0,T])$. For $I_1$, set $D:=B(0,1)\cup supp(\eta)$. We then have
\begin{align}\begin{split}\label{I1 weak}
&\iint_{\R^d\times \R^d} \{(M_\K[f^\K]-M[f])*\varphi_\K\} \eta\,dxdv\cr
&= \iiiint_{\R^{2d}\times \R^{2d}} (M_\K[f^\K]-M[f])(x-y,v-w)\varphi_\K(y,w) \eta(x,v)\,dxdydvdw\cr
&\le \|\eta\|_{L^\infty}\iint_{\mathbb{R}^d\times \mathbb{R}^d} |M_\K[f^\K]-M[f]|(x,v)\mathds{1}_{D}\,dxdv.
\end{split}\end{align}
On the other hand, it follows from \eqref{conv_ae} that
$$
(c_d \rho^{\K}_{f^\K})^{\frac 1d}    \to (c_d \rho_f)^{\frac 1d},\quad  u^{\K}_{f^\K}\to u_f \quad \mbox{a.e. on } E.
$$
Recalling the definition of $M$:
$$
M_\K[f^\K]=  \mathds{1}_{|u^{\K}_{f^\K}-v|^d\le c_d\rho^{\K}_{f^\K}}, \qquad  M[f]=  \mathds{1}_{|u_{f}-v|^d\le c_d\rho_{f}}.
$$
this implies that for each $v$ in the closure of $B(u_f, (c_d\rho_f)^{\frac 1d})$,
$$
M_\K[f^\K](v)\rightarrow 1,\quad \mbox{and otherwise}\quad M_\K[f^\K](v)\rightarrow 0\quad \mbox{a.e. on } E\quad \mbox{as}\quad \K\to 0,
$$
i.e. $M_\K[f^\K](v)$ converges to $M[f]$ a.e. on $E\times \mathbb{R}^d$. Moreover, we have from the $L^1$ bound of $f$ and $f^\K$ that for any $p\in(1,\infty)$,
\begin{align*}
\|M_\K[f^\K]-M[f]\|_{L^p}^p&\le C\iint_{\mathbb{R}^d\times \mathbb{R}^d} M_\K[f^\K]+M[f]\,dxdv\le C\int_{\mathbb{R}^d} \rho_{f^\K}+\rho_f\,dx <C.
\end{align*}
Thus we see that $|M_\K[f^\K]-M[f]|$ weakly converges to $0$ in $L^p$, which combined with \eqref{I1 weak} gives
$$
\iint_{E\times \R^d}I_1 \eta\,dxdv\rightarrow 0\quad\mbox{as}\quad\K\to0.
$$
On the other hand, on $E^c \times \R^d$, we estimate
\begin{align*}
\lt|\lim_{\K \to 0}\iint_{E^c \times \R^d} (M_\K[f^\K]*\varphi_\K-M[f]) \eta \,dxdtdv \rt| &\leq \|\eta\|_{L^\infty} \lim_{\K \to 0} \iint_{E^c \times \R^d} M_\K[f^\K] \,dxdtdv \cr
&\leq\lim_{\K \to 0}\int_{E^c} \rho_{f^\K}\,dxdt\cr
&=\int_{E^c} \rho_f \,dx dt\cr
&= 0.
\end{align*}
Thus we conclude
\begin{align*}
&\lim_{\K \to 0}\int_0^T\iint_{\R^d \times \R^d}  (M_\K[f^\K]*\varphi_\K-M[f]) \eta\,dxdvdt \cr
&\quad = \lim_{\K \to 0}\iint_{E \times \R^d}  \{(M_\K[f^\K]-M[f])*\varphi_\K\}  \eta\,dxdtdv + \lim_{\K \to 0}\iint_{E^c \times \R^d}  \{(M_\K[f^\K]-M[f])*\varphi_\K\}  \eta\,dxdtdv\cr
&\quad = \iint_{E \times \R^d}  (I_1+I_2)  \eta\,dxdtdv\cr
&\quad = 0.
\end{align*}
Finally, it only remains to deal with the alignment term:
 \begin{align*}
&\int_0^T \iint_{\R^d\times \R^d} f^\K F_{\phi}[f^\K]\cdot\nabla_v\eta\,dxdvdt\cr
&\quad =\int_0^T \iint_{\R^d\times \R^d} f^\K  \phi * (\rho_{f^{\K}} u_{f^{\K}}) \cdot\nabla_v\eta\,dxdvdt-\int_0^T \iint_{\R^d\times \R^d} f^\K \phi * \rho_{f^{\K}}v\cdot\nabla_v\eta\,dxdvdt.
\end{align*}
Since $f^\K \stackrel{\ast}{\rightharpoonup} f$ in $L^\infty (\R^d \times \R^d \times (0,T))$, we see that $f^\K \zeta $ converges weakly in $L^1 (\R^d \times \R^d \times (0,T))$ to $f\zeta$ for any test function $\zeta\in C_c(\R^d \times \R^d \times (0,T))$. Also, we have from \eqref{strong_comp} that
$$
\phi * \rho_{f^{\K}} u_{f^{\K}}\rightarrow  \phi * \rho_f u_f\quad\mbox{and}\quad \phi * \rho_{f^{\K}} \rightarrow   \phi *\rho_f \quad\mbox{a.e.}\quad\mbox{on}\quad \R^d\times\R^d\times(0,T) ,
$$
and these are  bounded uniformly in $\K$ thanks to Proposition \ref{prop_lin} and  the kinetic energy estimate. Therefore, we conclude that 
$$
\int_0^T \iint_{\R^d\times \R^d} f^\K F_{\phi}[f^\K]\cdot\nabla_v\eta\,dxdvdt \rightarrow \int_0^T \iint_{\R^d\times \R^d} f F_{\phi}[f]\cdot\nabla_v\eta\,dxdvdt
$$
as $\K \to 0$. This completes the proof.

%%%%%%%%%%%%%%%%%%%%%%%%%%%%%%%%%%%%%%%%%%%%%%%%%%%%%%%%%%%%%%%%%%%%%%%%%%%%%%%%%%%%%%%%%%%%%%%%%%%%%%%%%%%%%%%%%%%%%%%%%%%%%%%%%%%%%%%%%%%%%%
%
%
%    Acknowledgments
%
%
%%%%%%%%%%%%%%%%%%%%%%%%%%%%%%%%%%%%%%%%%%%%%%%%%%%%%%%%%%%%%%%%%%%%%%%%%%%%%%%%%%%%%%%%%%%%%%%%%%%%%%%%%%%%%%%%%%%%%%%%%%%%%%%%%%%%%%%%%%%%%%%%

\section*{Acknowledgments}
Y.-P. Choi and B.-H. Hwang were supported by National Research Foundation of Korea(NRF) grant funded by the Korea government(MSIP) (No. 2022R1A2C1002820).

%%%%%%%%%%%%%%%%%%%%%%%%%%%%%%%%

%	\section*{Acknowledgements}

%%%%%%%%%%%%%%%%%%%%%%%%%%%%%%%%%%%%%%%%%%%%%%%%%%%%%%%%%%%%%%%%%%%%%%%%%%%%%%%%%
%
%
%                        thebibliography
%
%
%%%%%%%%%%%%%%%%%%%%%%%%%%%%%%%%%%%%%%%%%%%%%%%%%%%%%%%%%%%%%%%%%%%%%%%%%%%%%%%%%

\end{document}